\documentclass[11pt]{amsart}
\usepackage{amscd}
\usepackage{amsfonts}
\usepackage{color}
\usepackage{amsmath,amssymb,amsthm,latexsym}
\usepackage{amscd}

\usepackage{multicol}
\usepackage{graphicx}
\newtheorem{theorem}{Theorem}[section]
\newtheorem{proposition}[theorem]{Proposition}
\newtheorem{definition}[theorem]{Definition}
\newtheorem{corollary}[theorem]{Corollary}
\newtheorem{lemma}[theorem]{Lemma}
\usepackage{amsmath}
\usepackage{amsfonts}
\usepackage{amsmath,amsthm}
\usepackage{amscd}
\usepackage[all]{xy}
\numberwithin{equation}{section}
\theoremstyle{remark}
\newtheorem{remark}[theorem]{Remark}
\newtheorem{example}[theorem]{\bf Example}
\newcommand{\R}{\mathbb{R}}
\newcommand{\C}{\mathbb{C}}
\newcommand{\D}{\mathbb{D}}

\newcommand{\FC}{\mathcal{C}}
 \newcommand{\dd}{\mathrm{d}}
\newcommand{\blue}{\textcolor{blue}}

\begin{document}

\title[Symmetric Willmore surfaces: the orientation reversing case]{\bf{On symmetric Willmore surfaces in spheres II: the orientation reversing case}}
\author{Josef F. Dorfmeister, Peng Wang }

\maketitle

\begin{abstract}
In this paper we provide a systematic treatment of Willmore surfaces with orientation reversing symmetries
and illustrate the theory by (old and new) examples. We apply our theory to isotropic Willmore two-spheres in $S^4$ and derive a necessary condition for such ( possibly branched) isotropic surfaces to descend to (possibly branched) maps from $\R P^2$ to $S^4$. The Veronese sphere and several other examples  of non-branched Willmore immersions from $\R P^2$ to $S^4$ are derived as an illustration of the general theory. The Willmore immersions of  $\R P^2$,  just mentioned and  different from the Veronese sphere, are new to the authors' best knowledge.
\end{abstract}
\vspace{2mm}



\tableofcontents

\section{Introduction}

This is the second paper of these authors concerning Willmore surfaces with symmetries. As stated in the title, we consider symmetries inducing anti-holomorphic automorphisms on the Riemann surface on which the Willmore surface is defined. Non--orientable Willmore surfaces
$f: M\rightarrow S^{n+2}$ can then be looked at as quotients of orientable Willmore surfaces invariant under some fixed point free anti-holomorphic involution $\mu$.

Our work is motivated by our interests in non--orientable Willmore surfaces. So far there are very few results on this topic. It is therefore natural to develop a systematic theory on Willmore surfaces with orientation reversing symmetries. In particular, there are two different types of orientation reversing symmetries: the ones with and the ones without a fixed point.  Since we aim at non--orientable surfaces, we will discuss mainly the latter case in this paper.

We will mainly follow the notation  of   \cite{DoWa1,DoWa11,DoWa-sym1}  and apply the basic results listed/obtained there to this paper for the  discussion of  the orientation reversing case.

As one might expect, the treatment of the orientation reversing case is somewhat more difficult than the treatment of the orientation preserving case. The basic technical reason for this is that the image of the conformal Gauss map is a four-dimensional Lorentzian subspace of $\R^{n+4}_1$ which carries naturally an orientation and is located inside the naturally oriented $R^{n+4}_1$ and that an orientation reversing symmetry reverses the orientation of the image of the conformal Gauss map. As a consequence we will need to keep track of orientations in addition to the features discussed in  \cite{DoWa-sym1}.

To achieve our goal we therefore recall and somewhat refine the basic approach
 to the discussion of Willmore surfaces presented in \cite{DoWa1}. This has actually nothing to do with our eventual goal to apply the loop group method for the discussion of Willmore surfaces: we refine and extend the classical Willmore theory to be able to discuss the orientation reversing symmetries. This will be the main contents of Section 2.

In Section 3, we consider mainly the general loop group description of (orientable) Willmore immersions admitting orientation reversing symmetries. This is slightly more technical than the theory for the orientation preserving case. There are always two fundamental steps: firstly, starting from a surface and to derive the loop group data, like meromorphic or holomorphic extended frames and potentials; secondly, starting from loop group data, like  the meromorphic or holomorphic extended frames and potentials and to construct a surface.

The perhaps most important and most interesting application of these results is the description of non--orientable Willmore surfaces in $S^{n+2}$ which forms the main content of Section 4.  Note that in this case, the symmetry is a finite   or infinite order symmetry without fixed point.
For instance, a non--orientable Willmore surface can be viewed as the quotient of an oriented Willmore surface $y$ ``by an orientation reversing involution"  $\mu^*y=y$, i.e. a symmetry $(\mu,I)$,  where $\mu$ is an anti-holomorphic automorphism without any fixed point satisfying $\mu^2=Id.$

Moreover, it is well--known that up to bi--holomorphic equivalence, there is only one such map on $S^2$ and none in the other simply-connected cases. Therefore, if one considers orientation reversing symmetries for surfaces defined on the unit disk or the complex plane  (and generating freely acting groups), then one needs to deal with infinite order anti-holomorphic automorphisms of these complex domains.

We will therefore divide up the discussion into the  cases where the Willmore immersion is defined on $\C$ or on $S^2$. The case where $\D$ is the domain of the Willmore immersion  is left to a separate investigation.
In Section 4.3, we apply the general theory to non--orientable Willmore surfaces and explain the various steps of the general construction procedure.
In particular we show that it is not possible to obtain non--orientable Willmore immersions of
$\R P^2$ from invariant potentials.

To illustrate these results, in Section 5, we prove an efficient and necessary condition to obtain isotropic Willmore two-spheres in $S^4$ invariant under the anti--holomorphic automorphism $\mu(z)=-\frac{1}{\bar{z}}$, hence constructing (possibly branched) Willmore immersions from $\R P^2$ to $S^4$. See Theorem \ref{thm-rp2} for details. It is easy to derive the well--known Veronese surface in $S^4$ in this way.  But we also  obtain
 several new examples of Willmore immersions from $\R P^2$ to $S^4$
with various Willmore energies, which are different from the one of the Veronese surface. These surfaces are in fact minimal in $S^4$ and, to the authors' best knowledge, they are the first concrete known  examples of minimal immersions  from $\R P^2$ to $S^4$ which are different from the Veronese surface.

 Another  type of non--orientable surfaces are the (doubly infinite) Moebius strips. \blue{ We also give a complete characterization of  (equivariant) Willmore Moebius strips in $S^{n+2}$  by applying the theory in this paper.
Since it involves a lot of other methods, this part will be shown in another publication \cite{Do-Wa-equ}.}

\begin{figure}[ht]
\centering
$
\begin{array}{ccc} 
\includegraphics[height=65mm]{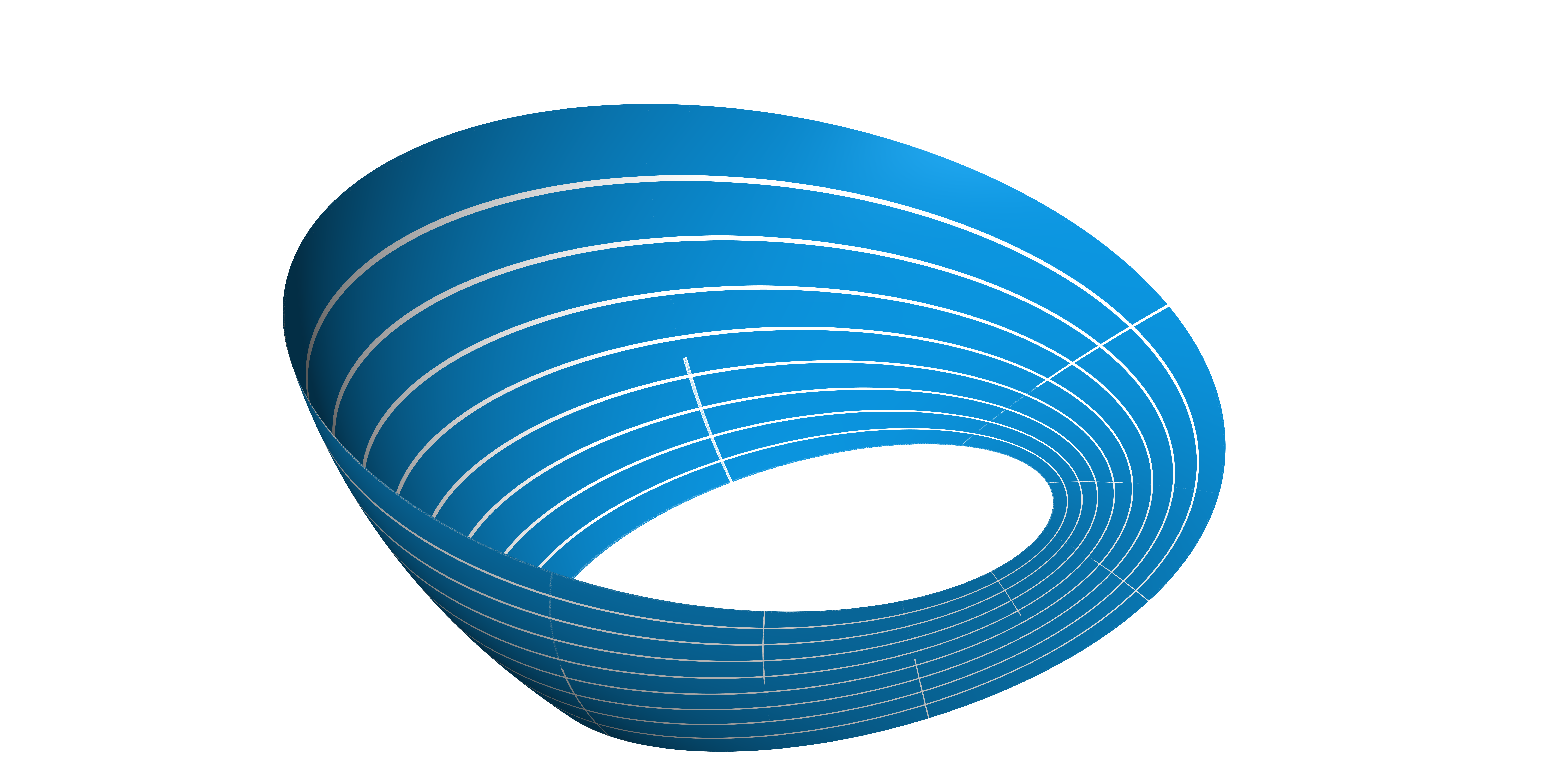}
\end{array}
$
\caption{Willmore Moebius strip computed with a numerical implentation of
DPW \cite{Brander} (with parameter X = bjorlingc(1,  0,  0,  -1/2,   -2-1/2) in \cite{Brander}).}
\label{figureexamples}
\end{figure}

 These examples indicate that this paper gives a blueprint for an efficient  discussion of all Willmore surfaces defined on non--orientable surfaces. While the technical details will stay difficult, the procedure is clear and straightforward.

We end this paper with an appendix including a proof of Theorem \ref{thm-rp2} which was separated out since it involves many technical computations. Moreover, we add some corrections and additions to \cite{DoWa-sym1}.

{ \em Throughout this paper we will use the notation introduced in \cite{DoWa-sym1}. In particular, by ``surface'' we mean ``conformal, branched surface'' unless the opposite is stated explicitly.} We would like to point out that for contractible Riemann surfaces $\D$ we will use the notion of ``frame" and ``extended frame" as usual, meaning globally defined and real analytic matrix functions on $\D$. But for $S^2$, by a ``frame" or ``extended frame" we  mean a matrix function defined and real analytic on $S^2\setminus \{\mbox{two points} \}. $ Then the loop group formalism, connecting surfaces to frames and connecting frames to ``potentials" (with meromorphic coefficient functions),  can be applied as usual. For more details we refer to Section 2.2.7 of this paper or Theorem 4.11 of \cite{DoWa1}.

\section{Review of Willmore surfaces and the loop group construction}

In this paper we discuss orientation reversing symmetries for  Willmore surfaces defined on  one of the simply connected Riemann surfaces: unit  disk,  complex plane or  Riemann sphere. We will use the notation $\D=$ for the unit disk as well as the complex plane. We use in all cases the procedure  outlined in \cite{DPW,DoWa1,DoWa11}.
\subsection{Willmore surfaces and their oriented conformal Gauss maps}

For the readers convenience  we briefly recall the  definition of the conformal Gauss map. We refer to  \cite{BPP,Ma,DoWa1,DoWa11} for details.

In \cite{BPP}, a natural and simple treatment of the conformal geometry of Willmore surfaces in $S^{n+2}$ is presented.
Here  we will use the same set--up. In particular, we will use the projective light cone model of the conformal geometry of
$S^{n+2}$ and discuss orientation  reversing conformal symmetries of Willmore surfaces in this model.

Let $\mathbb{R}^{n+4}_1$ denote Minkowski space, i.e. we consider $\mathbb{R}^{n+4}$ equipped with the Lorentzian metric
\[\langle x,y\rangle=-x_{0}y_0+\sum_{j=1}^{n+3}x_jy_j=x^t I_{1,n+3} y,\ \hspace{3mm}  I_{1,n+3}=diag(-1,1,\cdots,1).\]
Let $\mathcal{C}_+^{n+3}= \lbrace x \in \mathbb{R}^{n+4}_{1} |\langle x,x\rangle=0 , x_0 >0 \rbrace $
denote the forward light cone of $\mathbb{R}^{n+4}_{1}$.
It is easy to see that the projective light cone
$
Q^{n+2}=\{\ [x]\in\mathbb{R}P^{n+3}\ |\ x\in \mathcal{C}_+^{n+3}
\}$
with the induced conformal metric, is conformally equivalent to $S^{n+2}$.
Moreover, the conformal group of
$Q^{n+2}$ is exactly the projectivized orthogonal group $O^+(1,n+3)=O(1,n+3)/\{\pm1\}$ of
$\mathbb{R}^{n+4}_1$, acting on $Q^{n+2}$ by
\[
T([x])=[Tx],\,\,\ T\in O(1, n+3):=\{A\in Mat(n+4,\R)| A^tI_{1,n+3}A=I_{1,n+3}\}.
\]
 Here
$O^+(1,n+3):=\{A\in O(1, n+3)|  A \hbox{ preserves the forward light cone}\}.$

Let $y:M\rightarrow S^{n+2}$ be a conformal immersion from a Riemann surface $M$.
 Let $U\subset M$ be a contractible open subset. A local
lift of $y$ is a map $Y:U\rightarrow \mathcal{C}_+^{n+3} $ such
that \[[Y]=y.\] Two different local lifts differ by a scaling,
thus they are conformally equivalent to each other.
For our purposes we need to use canonical lifts, i.e. lifts satisfying
$|Y_z|^2 = \frac{1}{2}.$ Defining $\omega $ by $|y_z|^2 = \frac{1}{2}e^{2\omega}$  the canonical lift  (with respect
to $z$) is given by $Y = e^{-\omega} (1,y)$.
Noticing
$\langle Y,Y_{z\bar{z}}\rangle =-\langle Y_{z},Y_{\bar{z}}\rangle <0$, we see that
\begin{equation}
V={\rm Span}_{\mathbb{R}}\{Y,{\rm Re}Y_{z},{\rm Im}Y_{z},Y_{z\bar{z}}\}
\end{equation}
is an oriented rank--4 Lorentzian sub--bundle over $U$,
where the orientation is given by the obvious order of the given basis generating $V$.
For $\mathbb{R}^{n+4}_{1}$ we will always use the natural orientation given by the standard basis of $\R^{n+4}$.

Then there is a natural decomposition of the oriented trivial bundle $U\times \mathbb{R}^{n+4}_{1}=V\oplus V^{\perp}$, where $V^{\perp}$ is the orthogonal complement of $V$ with the induced natural orientation. Note that  $V$ and $V^{\perp}$,  and the orientation of $V$ are independent of the choice of $Y$
and $z$, and therefore are conformally invariant. In fact, we obtain a global conformally invariant bundle decomposition
$M\times \mathbb{R}^{n+4}_{1}=V\oplus V^{\perp}$. For any $p\in M$, we denote by $V_p$ the fiber of $V$ at $p$.

Given a canonical lift $Y$ we choose the frame $\{Y,{\rm Re}Y_{z},{\rm Im}Y_{z},N\}$ of
$V$, where $N \in\Gamma(V|_{U})$ is the unique section taking values in the forward light cone $\mathcal{C}_+^{n+3}$ and satisfies
\begin{equation}\label{eq-N}
N=2Y_{z\bar{z}}\mod Y,\mbox{ and } \langle N,Y_{z}\rangle=\langle N,Y_{\bar{z}}\rangle=\langle
N,N\rangle=0,~\langle N,Y\rangle=-1.
\end{equation}Next we define \emph{the conformal Gauss map} of $y$ (see, e.g. \cite{DoWa1}).
This conformal Gauss map of the immersion $y$ is basically given by
$p \mapsto V_p$.  But there are two additional features associated with this, namely
the $4$--dimensional Lorentzian subspace
$V|_p={\rm Span}_{\mathbb{R}}\{Y,{\rm Re}Y_{z},{\rm Im}Y_{z},Y_{z\bar{z}}\}_{p}$  has an orientation given by the ordered basis $\{Y,{\rm Re}Y_{z},{\rm Im}Y_{z},Y_{z\bar{z}}\}_{p}$ and the whole space $\R^{n+4}$ has, as always, its natural orientation.
In summary, the image of the conformal Gauss map is an oriented 4--dimensional Minkowski
space in the naturally oriented $\R^{n+4}$. It is well known and easy to verify that the set of all of such subspaces is the symmetric space
\begin{equation}
Gr_{1,3}(\mathbb{R}^{n+4}_{1}) = SO^+(1,n+3)/SO^+(1,3)\times SO(n).
\end{equation}
Note that here we use $SO^+(1,n+3)$, the matrices in $O^+(1,n+3)$ with determinant $1$, to describe the symmetric space, since it acts on $Gr_{1,3}(\mathbb{R}^{n+4}_{1})$ transitively.
\begin{definition} \label{def-gauss}
Let $M$ be a Riemann surface and $y:M\to S^{n+2}$ a conformal immersion.
 The  \emph{conformal Gauss map} of $y$ is defined by
\begin{equation}
\begin{array}{ccccc}
Gr_y : & M &\rightarrow & Gr_{1,3}(\mathbb{R}^{n+4}_{1}) &= SO^+(1,n+3)/SO^+(1,3)\times SO(n)\\
\ &  p & \mapsto & V_p &\ \\
\end{array}
\end{equation}
\end{definition}
 \begin{remark}\
\begin{enumerate}
\item  It is important to keep in mind  that the image $Gr_y (p) = V_p$ has the orientation discussed above.

\item Also note that our definition of $Gr_y$  implies that it maps to  the symmetric space $SO^+(1,n+3)/SO^+(1,3)\times SO(n)$ instead of $SO^+(1,n+3)/S(O^+(1,3)\times O(n))$, which is usually used in the literature about  Willmore surfaces.

\item Finally it is important to point out that $Gr_y$ depends on the conformal immersion $y$ as well as on the chosen complex structure of the Riemann surface $M$.
\end{enumerate}
\end{remark}

\begin{theorem}\cite{Bryant1984,Ejiri1988}
For every conformal immersion $y \rightarrow S^{n+2}$ the map $Gr_y$ is  a conformal map. Moreover,  $y$ is a Willmore surface if and only if $Gr_y$ is harmonic.
\end{theorem}

 Setting $Y_z=\frac{1}{2}(e_1-ie_2)$, $e_0 = \frac{1}{\sqrt{2}}(Y+N)$ and $\hat{e}_0 =\frac{1}{\sqrt{2}}(-Y+N)$, it is easy to verify that the ordered, linearly independent vectors $\{Y,{\rm Re}Y_{z},{\rm Im}Y_{z},Y_{z\bar{z}}\}$, $\{Y,N,e_1,e_2\}$ and $\{e_0,\hat e_0,e_1,e_2\}$ span the same ordered real vector space.
Hence we obtain
\begin{equation}\label{eq-gauss-map}
\begin{split}
Gr_y :=
{\rm Span}_{\mathbb{R}}\{Y,{\rm Re}Y_{z},{\rm Im}Y_{z},Y_{z\bar{z}}\}
=  {\rm Span}_{ \mathbb{R}} \{Y, N, e_1,e_2 \}= {\rm Span}_{ \mathbb{R}} \{e_0, \hat{e}_0,  e_1,e_2 \},
\end{split}
\end{equation}

Let $\psi_1,\cdots \psi_n$ be an orthonormal basis of $V^{\perp}$
inducing  the uniquely determined orientation of $V^{\perp}$.  Then a lift of $Gr_y$ into $SO^+(1,n+3)$ is given by (see Proposition 2.2 of \cite{DoWa1}):
\begin{equation}\label{eq-F}
F=\left(\frac{1}{\sqrt{2}}(Y+N),\frac{1}{\sqrt{2}}(-Y+N),e_1,e_2,\psi_1,\cdots,\psi_n\right)=\left(e_0,\hat{e}_0,e_1,e_2,\psi_1,\cdots,\psi_n\right).
\end{equation}

\begin{remark}
For later purposes we collect a few additional remarks.
\begin{enumerate}
\item The Lie algebra of $SO^+(1,n+3)$ is
\begin{equation}\label{eq-so(1,n+3)}
\mathfrak{so}(1,n+3)=\mathfrak{g}=\{X\in gl(n+4,\mathbb{R})|X^tI_{1,n+3}+I_{1,n+3}X=0\}.
\end{equation}
\item The symmetric space $ Gr_{1,3}(\mathbb{R}^{n+4}_{1})  = SO^+(1,n+3)/ SO^+(1,3)\times SO(n)$ is defined by the involution
\begin{equation}
\sigma(A):SO^+(1,n+3)\rightarrow SO^+(1,n+3),\sigma(A)=\hat{D}A\hat{D}^{-1},
\end{equation}
with $\hat{D}=I_{4,n}=$diag$(-1,-1,-1,-1,1,\cdots,1)$. Note that the isotropy group
is not the whole fixed point group of $\sigma$.
\begin{equation}
K = SO^+(1,3)\times SO(n)\neq \hbox{Fix}_{\sigma}=S(O^+(1,3)\times O(n)).
\end{equation}

\end{enumerate}
\end{remark}

\subsection{ Willmore surfaces with antiholomorphic symmetry}
Let $M$ be a Riemann surface and $\tilde{M}$ its universal cover.
Then $\tilde{M}$ is the unit disk or $\C$, or $S^2=\C\cup\{\infty\}$, together with a complex structure.

As in \cite{DoWa-sym1}, a symmetry of some Willmore surface
$y : M \rightarrow S^{n+2}$ consists of a pair of maps $(\mu, S)$ satisfying \[y\circ\mu(z)=y(\mu(z)) = S(y(z)),\]
where  $S$ is a conformal automorphism of $S^{n+2}$ which  maps $y(M)$ onto itself and $\mu: M \rightarrow M$ is a conformal  automorphism of $M$.
Similar to Theorem 3.1 of \cite{DoWa-sym1} one can show that such pairs of maps occur under natural assumptions.
Since the orientation preserving case is contained in \cite{DoWa-sym1}, we will assume from now on that $\mu$ is orientation reversing, i.e. anti--holomorphic relative to the complex structure given once and for all on $M$.
{\it The discussion in subsections $2.2.1 - 2.2.6$ will
be carried out for $M = \D$, an open and connected subset of $\C$.}
However, we will always assume that  global frames can be chosen on $\D$.

For a Willmore immersion $y$ with a symmetry $(\mu, S)$ and a local lift $Y$, we obtain (see e.g Theorem 3.1 of \cite{DoWa-sym1})
\begin{equation}\label{sym}
y\circ\mu(z) =[Y\circ\mu(z)]=[\hat{S} Y(z)]=  S(y(z))\ \hbox{ for all } \ z \in \D,
\end{equation}
where  $\hat S\in O^+(1,n+3)$ denotes the natural extension of $S$ to Minkowski space $\R^{n+4}_1$.

Moreover, using the conformal Gauss map as defined  above, one considers the map
$ z \rightarrow Gr_y(z)$ (see \eqref{eq-gauss-map}).
{\it To simplify notation we will use just $f$ for $Gr_y$ if no confusion can arise.}

If \eqref{sym} is satisfied, it is important to consider the conformal Gauss maps
 $f,$ $\hat{f}$ and $\tilde{f}$ of the Willmore surfaces $y,$ $\hat{y} =S y$ and $\tilde{y} = y \circ \mu$ respectively and to discuss their relations.
 In our approach to investigate Willmore surfaces the next step is to consider frames for the conformal Gauss map. So the relations between the frames $F$, $\hat{F}$ and $\tilde{F}$ associated with  $f,$ $\hat{f}$ and $\tilde{f}$  respectively need to be discussed.

\subsubsection{ The Willmore surface $\hat y= Sy$}

In this case it is straightforward to follow the definition of the conformal Gauss map given above and we obtain
\begin{equation}\label{V-hat}
\hat{V} =   {\rm Span}_{ \mathbb{R}} \{\hat{S}Y,  \hat{S}N,  \hat{S}e_1, \hat{S}e_2 \} =\hat{S} \hspace{1mm}{\rm Span}_{ \mathbb{R}} \{Y, N, e_1,e_2 \} = \hat{S} V.
\end{equation}
as vector spaces  in the naturally ordered $\R^{n+4}_1$ with the orientation given by the indicated ordered bases.
As a consequence we obtain
\begin{equation} \label{f-hat}
\hat f = \hat{S} f.
\end{equation}

\subsubsection{The Willmore surface $\tilde y = y \circ \mu$}
The general procedure implies (defining real valued functions $\tau$ and $\theta$ by
 $\frac{\partial\mu}{\partial\bar {z}}=e^{\tau+i\theta}$ ),
\begin{equation}
\tilde{Y}(z) = e^{-\tau} Y \circ \mu (z).
\end{equation}
From this we obtain
\begin{equation}
\tilde{Y}(z)_z = (e^{-\tau}Y\circ\mu)_z=-\tau_{z}e^{-\tau}Y\circ\mu+e^{-\tau}\cdot\frac{\partial\bar\mu}{\partial z}(Y_{\bar{z}})\circ\mu,
\end{equation}
and infer
\begin{equation*}
\begin{split}
Y\circ\mu&=e^{\tau}\tilde Y,\\
 e_1\circ\mu &=\tilde{e}_1 \cos\theta+\tilde{e}_2\sin\theta-\sqrt{2}(a\cos\theta+b\sin\theta)\tilde{Y},\\
 e_2\circ\mu &=\tilde{e}_1 \sin\theta-\tilde{e}_2\cos\theta-\sqrt{2}(a\sin\theta-b\cos\theta)\tilde{Y},\\
N\circ\mu&= e^{-\tau}\tilde{N}-\sqrt{2}a e^{-\tau}\tilde{e}_1-\sqrt{2} b e^{-\tau}\tilde{e}_2+(a^2+b^2)e^{-\tau}\tilde{Y},
\end{split}
\end{equation*}
with $ a=-Re(\sqrt{2}\tau_z),~~ b=Im(\sqrt{2}\tau_z).$
As a consequence, we obtain the matrix equation
\begin{equation} \label{relation}
 (e_0,\hat{e}_0,e_1,e_2)  \circ \mu =(\tilde{e}_0, \hat{\tilde{e}}_0, \tilde{e}_1, \tilde{e}_2) k_1,
\end{equation}
where  $e_0=\frac{1}{\sqrt{2}}(Y+N),$ $\hat{e}_0=\frac{1}{\sqrt{2}}(-Y+N)$, $\tilde{e}_0= \frac{1}{\sqrt{2}}(\tilde{Y}+\tilde{N}),$ $\hat{\tilde{e}}_0= \frac{1}{\sqrt{2}}(-\tilde{Y}+\tilde{N}),$  and
 \begin{equation}\label{eq-k1}
 k_1=\left(
                                   \begin{array}{cccc}
                                   \frac{ e^{\tau}+ e^{-\tau}(a^2+b^2+ 1) }{2}  & \frac{- e^{\tau}+e^{-\tau}(a^2+b^2 +1) }{2}   & -a\cos\theta-b\sin\theta  &  -a\sin\theta+b\cos\theta  \\
                                   \frac{- e^{\tau}-e^{-\tau}(a^2+b^2 -1) }{2}  &\frac{ e^{\tau}- e^{-\tau}(a^2+b^2-1) }{2}    &  a\cos\theta+b\sin\theta  & a\sin\theta-b\cos\theta \\
                                     -ae^{-\tau} &  -ae^{-\tau}   & \cos\theta & \sin\theta \\
                                     -be^{-\tau}  &  -be^{-\tau} & \sin\theta & -\cos\theta \\
                                   \end{array}
                                 \right).
 \end{equation}

It is straightforward to verify that $k_1\in O^+(1,3)$ and $\det k_1 = -1$. Therefore the vector space $\tilde{V} =  {\rm Span}_{ \mathbb{R}}  \{\tilde{Y}, \tilde{N}, \tilde{e}_1, \tilde{e}_2 \}$
is the same as $V\circ\mu={\rm Span}_{\mathbb{R}}\{Y,N,e_1,e_2\} \circ \mu$, but has the opposite orientation. Hence these two spaces correspond to different points in the Grassmannian $ Gr_{1,3}(\mathbb{R}^{n+4}_{1})  = SO^+(1,n+3)/ SO^+(1,3)\times SO(n)$.

However, there is a natural  involution
\begin{equation}\label{eq-star-op}
\ ^\star:  Gr_{1,3}(\mathbb{R}^{n+4}_{1})  \rightarrow  Gr_{1,3}(\mathbb{R}^{n+4}_{1}),
\end{equation}
 which maps a point  in the Grassmannian (i.e. an oriented vector space) into the same vector space but with the opposite orientation.
 This involution is the deck transformation of the two-fold covering
 \begin{equation}
 \begin{split}
  SO^+(1,n+3)/ SO^+(1,3)\times SO(n) \rightarrow
  SO^+(1,n+3)/S(O^+(1,3)\times O(n)).
\end{split}
 \end{equation}
 Thus we obtain
 \begin{proposition}
 The Gauss maps $\tilde{f}$ of $\tilde{y} = y \circ \mu$  and $f$ of $y$ are related by the equation
 \begin{equation}\label{eq-f-star}
 \tilde{f}^\star = f \circ \mu,\ \hbox{ or equivalently }
 \tilde{f} = (f \circ \mu)^{\star}.
 \end{equation}
 \end{proposition}

\subsubsection{ The relations between $\hat f$ and $\tilde f$}

Since $\hat{y}=\tilde{y}$ by definition \eqref{sym}, we have
\begin{equation}\label{f-hat-tilde}
\hat{f} =  \tilde{f}.
\end{equation}

Next we consider the relations between the corresponding frames.

 \subsubsection{A frame for the Willmore surface $\hat{y} = Sy$}

 Consider the frame $F$ of $y$ defined by \eqref{eq-F}, it is straightforward to see that $\hat F=\hat S F$ is a frame of $\hat{f}$ when $\det \hat S=1$. When $\det\hat S=-1$, $\hat F$ takes values in $O^+(1,n+3)$, and  not in $SO^+(1,n+3)$. But since the first four columns of $\hat F$ still provide $\hat f$ with the same orientation, we see that one only needs  to change the orientation of the remaining $n$ columns. This proves the following proposition
 \begin{proposition}
 With the notation introduced above we consider the surface  $\hat{y}=[\hat Y]=[\hat S Y]= S y$. Then the following statements hold
\begin{enumerate}
\item If $\det \hat{S} = 1,$ then
 \begin{equation}
 \hat {F} = \hat{S} F
 \end{equation}
 is a frame for $\hat y$.
\item If $\det \hat{S} = -1,$ then
 \begin{equation}
 \hat{F} = \hat{S} F P_2
 \end{equation}
 is a frame for $\hat y$, where
  $P_2=\hbox{diag}(I_4,-1,I_{n-1})\in O^+(1,n+3)$.
 \end{enumerate}
 \end{proposition}

 Note , the last statement simply says that we have changed the the original orthonormal bases of $V^\perp$ to another one which gives the whole (ordered) basis the same orientation as $\mathbb{R}^{n+4}_{1}$.

\subsubsection{A frame for the Willmore surface $\tilde{y} = y \circ \mu$}

In this case we know by \eqref{eq-f-star} that  $\tilde{f}=(f\circ\mu)^{\star}$ holds.
Thus we consider the frame
$\tilde{F} =(\tilde{e}_0,\tilde{\hat{e}}_0,\tilde{e}_1,\tilde{e}_2,\tilde{\psi}_1,\cdots,\tilde{\psi}_n)$
of $\tilde f$ from and compare it to the naturally formed matrix
\[ F\circ\mu = (e_0 \circ \mu ,\hat e_0\circ \mu , e_1 \circ \mu , e_2 \circ \mu, \psi_1 \circ \mu \cdots \psi_n \circ \mu)\in SO^+(1,n+3).\]

First we recall equation \eqref{relation}. This equation implies that the first four columns of $\tilde{F}$ and of $F \circ \mu$ span the same vector space, but induce opposite orientations, since $\det k_1 = -1.$ As a consequence, the vector spaces spanned by the last n columns of $\tilde{F}$ and of $F \circ \mu$ also are  equal and also have opposite orientation. Multiplying $ \tilde{F} $ on the right by $P_1P_2$  where
$P_1=\hbox{diag}(1,1,1,-1,,I_n)\in O^+(1,n+3)$ and $P_2$ is as above, we change the orientation of the vector spaces spanned by the first four and the last n columns so that they coincide with the corresponding oriented vector spaces. Consider now the matrices $\tilde{k}_1 = diag(k_1,I_n)P_1$ and
  $\tilde{k}_2 = diag(I_4, k_2)P_2,$ where $k_2$ maps the vectors  $\tilde{\psi}_1,\cdots,\tilde{\psi}_n$ to the vectors $\psi_1, \cdots, \psi_n$.
  Then $\tilde{k}_1, \tilde{k}_2 \in SO^+(1,3)\times SO(n)$ and we obtain

  \begin{proposition} Consider  the surface $ \tilde{y} = y \circ \mu : \D \rightarrow S^{n+2}$.
 Then each frame $\tilde{F}$ for $\tilde{f}$ is related to $F \circ \mu$ in the form
 \begin{equation}
  \tilde{F} = F \circ \mu \tilde k^{-1} P_1 P_2,
 \end{equation}
 where $\tilde{k} =  \tilde k_1  \tilde k_2 $ with $ \tilde k_1 = diag(k_1, I_n)P_1$ and $ \tilde{k}_2 = diag(I_4, k_2)P_2 $, just defined above,  are contained in $ SO^+(1,3)\times SO(n)$.
 \end{proposition}

\subsubsection{The frame equation for the symmetry $(\mu, S)$}
Since the symmetry relation \eqref{sym} implies $\tilde{f} = \hat{f}$, for the frames this
implies  $\tilde{F}=\hat{F} \check{k}$ with $\check{k}=diag\{\check{k}_1,\check{k}_2\} \in SO^+(1,3) \times SO(n)$. Summing up, we obtain
\begin{theorem}\label{th-symm-f} Let $y: \D \rightarrow  S^{n+2}$ be a Willmore immersion and $\left(\mu , S\right)$  an orientation--reversing symmetry. Then  \eqref{sym} holds.
Moreover,
\begin{enumerate}
\item There exists some real function $\tau$ such that $e^{\tau}=\left|\frac{\partial\mu}{\partial\bar {z}}\right|$  and
$e^{-\tau}Y\circ\mu=\hat{S}Y.$
\item The conformal Gauss map $f$ of $y$ satisfies the equation
\begin{equation}\label{eq-sym}
 (f \circ\mu)^{\star} = \hat{S}  f,  \hspace{2mm} \mbox{equivalently} \hspace{2mm}
 f \circ\mu = \hat{S}  f^\star
\end{equation}
\item
Let $F$ denote the moving frame associated with $f$ as defined in \eqref{eq-F}. Then there exists some $\hat k=diag(\hat{k}_1,\hat{k}_2): \D \rightarrow SO^+(1,3)\times SO(n)$ such that
\begin{enumerate}
\item If $\det \hat{S} = 1,$ then
 \begin{equation}\label{eq-Frame-1}
 F\circ\mu \hat k^{-1} P_1 P_2= \hat S F.
 \end{equation}
\item If $\det \hat{S} = -1,$ then
 \begin{equation}\label{eq-Frame-2}
 F\circ\mu  \hat k^{-1} P_1 P_2= \hat S FP_2.
 \end{equation}
 \end{enumerate}
\end{enumerate}
\end{theorem}

\begin{remark}\
 \begin{itemize}
 \item Note that both, $\hbox{Ad}(P_1)$ and  $\hbox{Ad}(P_2)$, preserve $SO^+(1,3)\times SO(n)$. Therefore the order in the multiplication $\hat k P_1 P_2$ is not essential.
\end{itemize}
\end{remark}

\subsubsection{The $S^2$ case} \label{section-s2}
 For the case of $M = S^2$, the conformal Gauss map $f$  of a Willmore immersion $y$ does not have a global frame. As in   \cite{DoWa1} and  \cite{DoWa-sym1} we choose two charts $\D=S^2\setminus\{z_0\}$ with $z_0=0$ or $z_0=\infty$.
Then $\D$ is contractible whence on $\D$ there exists  a global frame $F$ for $f$.

Note that for any orientation reversing symmetry $(\mu,S)$ of $y$ and $z_0$ as above we have two possibilities, $\mu(z_0)=z_0$ or $\mu(z_0)\neq z_0$.
For the first case, Theorem \ref{th-symm-f} holds without any changes.
For the second case, $\mu$ is defined on $\tilde\D=S^2\setminus\{z_0,\mu(z_0)\}$ and
$\mu:\tilde\D\rightarrow \tilde\D$ is an anti-holomorphic automorphism. As a consequence, Theorem \ref{th-symm-f}  hold on $\tilde\D$.

In later applications of our theory to the case of orientation reversing symmetries of Willmore surfaces $y: M = S^2  \rightarrow S^{n+2}$ we will always apply an argument like the one just given.

\subsection{ The loop group method}

Here we recall briefly the loop group construction of Willmore surfaces in spheres introduced in \cite{DoWa1} (see also \cite{DoWa-sym1}).
The basic idea has two steps. Firstly there is a relationship between Willmore surfaces and harmonic maps into a symmetric space $G/K$.
Secondly there is a way to describe these harmonic maps by some special meromorphic or holomorphic 1-forms, via the Birkhoff decomposition and the Iwasawa decomposition of the loop group associated with $G$.

  \subsubsection{Loop groups, decomposition theorems, and the loop group method}
 First we briefly recall the DPW construction for harmonic maps. Let $G$ be a connected, real, semi-simple non-compact matrix Lie group and let $G/K$ be the inner symmetric space defined by  the involution $\sigma: G\rightarrow G$, with $Fix^{\sigma} G\supseteq K \supseteq (Fix^{\sigma} G )^\circ$, where $H^\circ$ denotes the identity component of the group $H$.

Let $\mathfrak{g}$, $\mathfrak{k}$ be the Lie algebras of $G$ and $K$ respectively and let $\mathfrak{g}=\mathfrak{k}\oplus\mathfrak{p}$ be the Cartan decomposition induced by $\sigma$, with
$ [\mathfrak{k},\mathfrak{k}]\subset\mathfrak{k},
~~~ [\mathfrak{k},\mathfrak{p}]\subset\mathfrak{p}, ~~~
[\mathfrak{p},\mathfrak{p}]\subset\mathfrak{k}.$
Let $\pi:G\rightarrow G/K$ denote the projection of $G$ onto $G/K$.
Let  $\mathfrak{g^{\mathbb{C}}}$ be the complexification of $\mathfrak{g}$ and $G^{\mathbb{C}} $  the connected complex matrix  Lie group with Lie algebra $\mathfrak{g^{\mathbb{C}}}$.
Setting $\tau(g)=\bar g$ for $ g \in G^{\mathbb{C}}$ yields $G=(Fix^{\tau}G^{\C})^{\circ}$.
Moreover, $\sigma$ extends to
the complexified Lie group $G^\C$ and  commutes with $\tau$. Then  $K^{\mathbb{C}} = \hbox{Fix}^{\sigma}(G^{\mathbb{C}})^{\circ}$ denotes  the smallest (connected) complex subgroup of $G^{\C}$ containing $K$.

 Let $\mathcal{F}:M\rightarrow G/K$ be a conformal harmonic map from a connected Riemann surface $M$.
Let $U\subset M$ be an open contractible subset.
Then there exists a frame $F: U\rightarrow G$ such that $\mathcal{F}=\pi\circ F$. One has the Maurer--Cartan form $\alpha=F^{-1} \dd F$ and the Maurer--Cartan equation $\dd \alpha+\frac{1}{2}[\alpha\wedge\alpha]=0.$
Moreover, decomposing $\alpha$ with respect to $\mathfrak{g}=\mathfrak{k}\oplus\mathfrak{p}$ and  the complexification $T^*M^{\mathbb{C}}=T^*M'\oplus T^*M''$, we obtain
\[\alpha=\alpha_{\mathfrak{p}}'+\alpha_{ \mathfrak{k}  } +\alpha_{\mathfrak{p}}'', \hbox{ with }
\alpha_{\mathfrak{k  }}\in \Gamma(\mathfrak{k}\otimes T^*M),\
\alpha_{ \mathfrak{p }}'\in \Gamma(\mathfrak{p}^{\C}\otimes T^*M'),\
\alpha_{ \mathfrak{p }}''\in \Gamma(\mathfrak{p}^{\C}\otimes T^*M'').\] Set
 \begin{equation}\label{spec-par}
\alpha_{\lambda}=
\lambda^{-1}\alpha_{\mathfrak{p}}'+\alpha_{\mathfrak{k}}+
\lambda\alpha_{\mathfrak{p}}'', \hspace{5mm}  \lambda\in S^1.
\end{equation}

It is well--known  (\cite{DPW}) that the map  $\mathcal{F}:M\rightarrow G/K$ is harmonic if and only if
\begin{equation}\label{integr}\dd
\alpha_{\lambda}+\frac{1}{2}[\alpha_{\lambda}\wedge\alpha_{\lambda}]=0,\ \ \hbox{for all}\ \lambda \in S^1.
\end{equation}
 From equation \eqref{integr} we infer that there exists a solution $F(z,\bar z,\lambda)$  to the equation $\dd F(z,\bar z,\lambda)= F(z,\bar z,\lambda)\alpha_{\lambda}$ on $U \subset M.$
 Such a solution is uniquely determined, if we impose an
initial condition  $F(z_0,\bar z_0,\lambda)=F_0(\lambda)\in \Lambda G_{\sigma},~z_0 \in U$, where $z_0$ is chosen arbitrarily in $U$. The solution   $F(z, \bar z, \lambda)$
is called the {\em extended frame} of the harmonic map $\mathcal{F}$ (normalized at the base point $z=z_0$).
Set $\mathcal{F}_{\lambda}=F(z,\bar z,\lambda)\mod K$. Then  $\mathcal{F}_{\lambda}$ is harmonic for every $\lambda \in S^1$ and this family of harmonic maps will be called the ''associated family of the harmonic map''
$\mathcal{F}=\mathcal{F}_{\lambda}|_{\lambda=1}.$
Clearly,  $F(z,\bar z,\lambda)$ is a local lift of $\mathcal{F}_{\lambda}$.

Recall that the twisted loop groups of $G$ and $G^{\mathbb{C}}$  are defined as follows:
\begin{equation*}
\begin{array}{llll}
\Lambda G^{\mathbb{C}}_{\sigma} &=\{\gamma:S^1\rightarrow G^{\mathbb{C}}~|~ ,\
\sigma \gamma(\lambda)=\gamma(-\lambda),\lambda\in S^1  \},\\[1mm]

\Lambda G_{\sigma}  &=\{\gamma\in \Lambda G^{\mathbb{C}}_{\sigma}
|~ \gamma(\lambda)\in G, \hbox{for all}\ \lambda\in S^1 \},\\[1mm]

\Lambda^{-} G^{\mathbb{C}}_{\sigma}  ~&=
\{\gamma\in \Lambda G^{\mathbb{C}}_{\sigma}~
|~ \gamma \hbox{ extends holomorphically to } |\lambda|>1 \cup\{\infty\} \},\\[1mm]
\Lambda_{*}^{-} G^{\mathbb{C}}_{\sigma} ~&=\{\gamma\in \Lambda G^{\mathbb{C}}_{\sigma}~
|~ \gamma \hbox{ extends holomorphically to } |\lambda|>1 \cup\{\infty\},\  \gamma(\infty)=e \},\\[1mm]

\Lambda_{L}^{-} G^{\mathbb{C}}_{\sigma} ~&=\{\gamma\in \Lambda G^{\mathbb{C}}_{\sigma}~
|~ \gamma \hbox{ extends holomorphically to } |\lambda|>1 \cup\{\infty\},\  \gamma(\infty)\in L \},\\[1mm]

\Lambda^{+} G^{\mathbb{C}}_{\sigma}  &=\{\gamma\in \Lambda G^{\mathbb{C}}_{\sigma}~
|~ \gamma \hbox{ extends holomorphically to the disk} \hspace{1mm} |\lambda|<1,  \},\\

\Lambda_{*}^{+} G^{\mathbb{C}}_{\sigma}  &=\{\gamma\in \Lambda G^{\mathbb{C}}_{\sigma}~
|~ \gamma \hbox{ extends holomorphically to the domain } |\lambda|<1,\  \gamma(0)=e \},\\[1mm]

\Lambda_L^{+} G^{\mathbb{C}}_{\sigma}  &=\{\gamma\in \Lambda G^{\mathbb{C}}_{\sigma}~
|~ \gamma \hbox{ extends holomorphically to the disk} \hspace{1mm} |\lambda|<1,  \gamma(0)\in L\},\\

\end{array}\end{equation*}
where $L\subset K^{\C}$ is a subgroup. When $L=K^{\C}$, we denote the corresponding loop group  by $\Lambda_\mathcal{C}^{+} G^{\mathbb{C}}_{\sigma}$.

Now we restrict to the loop groups related with Willmore surfaces \cite{DoWa1,DoWa-sym1}.
The  Iwasawa decomposition states \cite{DPW,DoWa1} that there exists a closed, connected solvable subgroup $S \subseteq K^\C$ such that
the multiplication $\Lambda G_{\sigma}^{\circ} \times \Lambda^{+}_S G^{\mathbb{C}}_{\sigma}\rightarrow
\Lambda G^{\mathbb{C}}_{\sigma}$ is a real analytic diffeomorphism onto the open subset
$ \Lambda G_{\sigma}^{\circ} \cdot \Lambda^{+}_S G^{\mathbb{C}}_{\sigma}  \subset(\Lambda G^{\mathbb{C}}_{\sigma})^{\circ}$.
The  Birkhoff decomposition states that
the multiplication $\Lambda_{*}^{-} {G}^{\mathbb{C}}_{\sigma}\times
\Lambda^{+}_\FC {G}^{\mathbb{C}}_{\sigma}\rightarrow
\Lambda {G}^{\mathbb{C}}_{\sigma}$ is an analytic  diffeomorphism onto the
open and dense subset $\Lambda_{*}^{-} {G}^{\mathbb{C}}_{\sigma}\cdot
\Lambda^{+}_\FC {G}^{\mathbb{C}}_{\sigma}$ {\em ( big Birkhoff cell )}.

Now we state the DPW construction of Willmore surfaces in the spirit of  \cite{DPW}.

\begin{theorem}\label{thm-DPW}\cite{DPW}, \cite{DoWa1}, \cite{Wu}.
Let $\D$ be  the unit disk or $\C$ itself, with complex coordinate $z$ and fix a base point $z_0 \in \D$.
\begin{enumerate}
\item
Let $\mathcal{F}: \D \rightarrow G/K$ be a harmonic map with an extended frame $F(z,\bar z,\lambda)$ satisfying  $F(z_0,z_0,\lambda)=e$.
Then there exists a discrete subset $\mathcal{S}\subset \D$ such that for all $z \in \D \setminus \mathcal{S}$ there exists the Birkhoff decomposition
\[F(z,\bar z,\lambda)=F_-(z,\lambda)  F_+(z,\bar z,\lambda)~ \hbox{ with }~ F_+(z,\bar z,\lambda):\D \setminus \mathcal{S} \rightarrow\Lambda^{+}_{\FC} G^{\mathbb{C}}_{\sigma},\]
such that
 $ F_-(z,\lambda):\D \setminus \mathcal{S} \rightarrow\Lambda_{*}^{-} G^{\mathbb{C}}_{\sigma}
$ is meromorphic in $z $ on $ \D$ and satisfies $F_-(z_0, \lambda) = e.$ Moreover, its Maurer--Cartan form is of the form
\[\eta= F_-(z,\lambda)^{-1} \dd F_-(z,\lambda)=\lambda^{-1}\eta_{-1}(z)\dd z\]
with $\eta_{-1}(z)$ independent of $\lambda$. $\eta$ is called the {\bf normalized potential} of $\mathcal F$.

\item Conversely,  Let $\eta$ be a  $\lambda^{-1}\cdot\mathfrak{p}^{\mathbb{C}}-\hbox{valued}$ meromorphic $(1,0)-$ form with $F_-(z,\lambda)$ a solution to $F_-(z,\lambda)^{-1} \dd F_-(z,\lambda)=\eta, ~ F_-(z_0,\lambda)=e$
which is meromorphic on $\D$.  Then  there exists an open subset $\D_{\mathcal{I}}$  of $\D$  such that for all $z \in \D_{\mathcal{I}}$ we have an Iwasawa decomposition
\[F_-(z,\lambda)=\tilde{F}(z,\bar z,\lambda) \tilde{F}_+(z,\bar z,\lambda)^{-1}, \]
$\hbox{with } \tilde{F}(z,\bar z,\lambda)\in\Lambda G_{\sigma},\  \tilde{F}_+(z,\bar z,\lambda)\in\Lambda^{+}_\FC {G}^{\mathbb{C}}_{\sigma},$ $\tilde{F}(z_0,z_0,\lambda) = e$ and $ \tilde{F}_+(z_0,z_0, \lambda) = e.$
 Then $\tilde{F}(z,\bar z,\lambda) $ is an extended frame  of some harmonic map from
 $\D_{\mathcal{I}}$  to $G/K$ satisfying   $\tilde{F}(z_0,z_0,\lambda)=e$. Moreover, the two constructions above are inverse to each other, if throughout the normalizations at the base point $z_0$ are used.
\end{enumerate}

\end{theorem}

In many applications it is, for different reasons, more convenient to use potentials which have a Fourier expansion containing more than one power of $\lambda$.
As a matter of fact, when permitting many ( actually infinitely many) powers of $\lambda$,  one can even obtain holomorphic coefficients (at least in the case of a non-compact domain).

\begin{theorem}\cite{DPW}, \cite{DoWa1}.  We retain the notation introduced above.
\begin{enumerate}
\item There exists some  $V_+:\D \rightarrow  \Lambda^{+} G^{\mathbb{C}}_{\sigma} $ such that $C(z,\lambda) =
F(z,\bar z,\lambda) V_+(z,\bar z,\lambda)  $ is holomorphic in $z\in\mathbb{D}$ and in $\lambda \in \mathbb{C}^*$.
The Maurer--Cartan form $\tilde\eta = C^{-1} \dd C$ is a holomorphic $(1,0)-$form on $\D$ and  $\lambda \tilde\eta$ is holomorphic for $\lambda \in \C$.

\item Conversely, Let $\tilde\eta\in\Lambda\mathfrak{g}^{\C}_{\sigma}$ be a holomorphic  $(1,0)-$ form such that $\lambda \tilde\eta$ is holomorphic for $\lambda\in\C$, then by the same steps as in Theorem \ref{thm-DPW} we obtain a harmonic map $\mathcal{F}: \D\rightarrow G/K$.
\end{enumerate}
\end{theorem}

\begin{remark}
$C(z,\lambda)$ and $\tilde\eta$ are called a {\em holomorphic extended frame and a holomorphic potential} of the harmonic map $\mathcal{F}$ respectively. Note that $C(z,\lambda)$ and $\tilde\eta$ are not unique for a harmonic map. For instance, all holomorphic gauges of  $C(z,\lambda)$ by  some holomorphic $g_+(z, \lambda)$ satisfying $g_+(z_0, \lambda) = e$ yield new extended holomorphic frames/potentials, but these all will produce the same  surface.
\end{remark}

\subsubsection{The case of $S^2$}  In this subsection we have considered so far only contractible Riemann surfaces $\D$, i.e. the unit disk and the complex plane. We will now address  the case  of $S^2$. This case needs a bit of interpretation.  We refer to \cite{DoWa1} for a detailed discussion. We begin by recalling that in \cite{DoWa1} we have shown that every Willmore sphere can be generated from some meromorphic potential $\eta$ by
an application of the usual DPW procedure. This means, one considers a solution $F_-$ to the ode $\dd F_- = F_- \eta$ and can assume that it is meromorphic on $S^2$. One can also assume that $F_-(z_0,\lambda) =I$ at some base point $z_0$. Then one performs an Iwasawa decomposition, at least locally around $z_0$. The corresponding ``real factor'' is called ``extended frame". However, due to topological restrictions, it necessarily
 has (w.l.g. at most two) singularities on $S^2$.
This does not affect its role as ``extended frame" in the sense of the loop group theory. In particular, by projection to the symmetric target space $G/K$ of the conformally harmonic map of a Willmore immersion discussed here, also this singular extended frame induces some conformally harmonic map and as in the case of contractible simply--connected covers, this harmonic map induces some Willmore surface.
For this reason, from here on we will not distinguish in our discussion Willmore surfaces according to what  simply--connected cover they have.

\section{Willmore surfaces with orientation reversing symmetries}

 In this section, we will consider the orientation reversing symmetries $(\mu,  S)$ \eqref{sym} of a Willmore surface $y$ in terms of their loop group data. To this end, we will first discuss the behaviour of the extended frames under such symmetries. Then we will derive a formula satisfied by a normalized potential if the corresponding Willmore surface admits the orientation reversing symmetry $(\mu,S)$. Conversely, we will show that if one starts from a normalized potential which satisfies the condition just mentioned, then one will obtain a harmonic map with symmetry. Moreover,
 if this harmonic map is the conformal Gauss map of a Willmore surface, then the symmetry of the harmonic map induces a symmetry of the Willmore surface.

\subsection{Transformation formulas for  extended frames and potentials}
In Theorem \ref{th-symm-f}  we have given formulas for the frames of a Willmore surface with orientation reversing symmetry. Following the loop group approach for the description of all harmonic maps we
 now  need to consider maps depending on $z$ and on the loop parameter $\lambda$.

Introducing the loop parameter  (see Section 2.3 above or Section 4.1 of \cite{DoWa-sym1}) we can translate the results of Theorem \ref{th-symm-f} immediately
into the loop group setting ( For a map $\Phi(z,\bar{z},\lambda)$, we will write
$\Phi(\mu(z),\overline{\mu(z)},\lambda)$ to express $\Phi(z,\bar{z},\lambda) \circ \mu$.):

\begin{theorem}\label{th-symm-n-1}
Let $\D$ denote the unit disk or the complex plane and $y: \D \rightarrow S^{n+2}$ a Willmore surface. Let $z_0$ denote a base point in $\D$.
Let $(\mu , S )$ be
an orientation reversing symmetry of the harmonic conformal Gauss map $f: \D \rightarrow SO^+(1,n+3)/ SO^+(1,3)\times SO(n)$ of the Willmore surface $y$. Let $F(z,\bar{z},\lambda )$ be the extended frame of $f$ satisfying the initial condition $F(z_0,z_0,\lambda)=I$. Then:
 \begin{enumerate}
 \item There exists some ${M}(\lambda) \in \Lambda O^+(1,n+3)_{\sigma}$ and $ k (z,\bar{z})\in O^+(1,3)\times O(n)$
such that
\begin{equation} \label{transF}
F(\mu(z), \overline{\mu(z)}, \lambda) = M(\lambda) \cdot  F(z, \bar{z}, \lambda^{-1})  \cdot   k (z, \bar{z}).
\end{equation}
Here the gauge matrix $k$ is a diagonal block matrix of the form  $k = diag(k_1 , k_2)$, independent of $\lambda$, with $k_1 \in O^+(1,3)$ and
$ k_2 \in O(n)$. We always have $\det  k_1  = -1$ and $\det k_2=-\det \hat S$.
 Moreover, the matrix $\hat{S}$ in \eqref{sym}  is equal to $M(\lambda = 1)$ and the matrix $k_1$ has the explicit representation  \eqref{eq-k1}.
\item Moreover, for the associated family $f(z,\bar{z},\lambda)$, we have
\begin{equation} \label{eq-trans-f}
f(\mu(z), \overline{\mu(z)}, \lambda) = M(\lambda) f(z, \bar{z}, \lambda^{-1})^{\star}.
\end{equation}
\end{enumerate}
\end{theorem}

\begin{proof}From \eqref{eq-Frame-1} and \eqref{eq-Frame-2} we know
\begin{equation}\label{Fmu2}
F(\mu (z), \overline{ \mu (z)}) = \hat{S} F(z, \bar{z})  k (z, \bar{z}),
\end{equation}
where $k=P_2P_1\hat{k}$ if $\det\hat S=1$ and $k=P_1\hat{k}$ if $\det\hat{S}=-1$.
Note, in both cases, $ \hat{k}\in K $ and $k\in Fix_{\sigma} O^+(1,n+3)$.  Consider the Maurer--Cartan form  \[\alpha (z, \bar{z}) =  F(z, \bar{z})^{-1} \dd  F(z, \bar{z})\] and expand
\begin{equation*}
\alpha (z, \bar{z}) = \alpha' (z, \bar{z})_{\mathfrak{p}}+ \alpha (z, \bar{z})_{\mathfrak{k}} + \alpha'' (z, \bar{z})_{\mathfrak{p}}.
\end{equation*}
Since conjugation by $ k $ leaves $\mathfrak{k}$  and $\mathfrak{p}$ invariant, \eqref{Fmu2} yields
\begin{equation}\label{mustaralpha}
\mu^*\alpha' _{\mathfrak{p}} + \mu^*\alpha _{\mathfrak{k}} + \mu^*\alpha'' _{\mathfrak{p}} =  k ^{-1} \alpha' _{\mathfrak{p}}  k  +  k ^{-1} \alpha_{\mathfrak{k}}  k  +  k ^{-1}  \dd k  +
 k ^{-1} \alpha''_{\mathfrak{p}}  k .
\end{equation}
Note that $\mu^*\alpha' _{\mathfrak{p}}$ is a $(0,1)-$form and
$\mu^*\alpha'' _{\mathfrak{p}}$ is a $(1,0)-$form.
Thus \eqref{mustaralpha} is equivalent to
\begin{equation}\label{compare}
\mu^*\alpha' _{\mathfrak{p}} =  k ^{-1} \alpha'' _{\mathfrak{p}}  k ,
\hspace{2mm}
\mu^*\alpha _{\mathfrak{k}} =  k ^{-1} \alpha_{\mathfrak{k}}  k  +  k ^{-1}  \dd k , \hspace{2mm} \mbox{and} \hspace{2mm}
\mu^*\alpha'' _{\mathfrak{p}} =  k ^{-1} \alpha' _{\mathfrak{p}}  k .
\end{equation}
In view of \eqref{compare} it is now  easy to see that the Maurer--Cartan form $\alpha(z, \bar{z},\lambda)$ of the extended frame $F(z, \bar{z},\lambda)$ of the conformal Gauss map of $f$ satisfies
\begin{equation}\label{alphatrans}
\mu^* \alpha(z, \bar{z},\lambda) =  k ^{-1} \alpha(z, \bar{z}, \lambda^{-1})  k  +  k ^{-1}  \dd k .
\end{equation}
This implies for the extended frame (satisfying $F(z_0, \bar{z}_0,\lambda) = I$)
\[ F(\mu(z), \overline{\mu(z)}, \lambda) = M  (\lambda) F(z, \bar{z}, \lambda^{-1})  k (z, \bar{z}),\]
where a priory $ M (\lambda)\in \Lambda O^+(1,n+3)_{\sigma}$ and $ M (\lambda=1)=\hat{S}$.
Since $\det F=1$, we see that $\det\hat{S}=\det M  (\lambda)=\det  k (z, \bar{z})=\pm1$.
 The last statement follows from (\ref{eq-sym}).
\end{proof}

 \begin{remark}\
The fact $\det k_1=-1$ means that $k$ does not take values in $K=SO^+(1,3)\times SO(n)$.
So we can not consider the quotient by $K$ straightforwardly. In other words, as we have discussed in Section 2.2, since $\det k_1=-1$, it changes the orientation of the oriented $4-$dim Lorentzian subspace.
Moreover, since $\det \hat{S}=\det k$, we also need to keep in mind  the signature of $\det \hat{S}$. As a consequence,  in view of $\det \hat{S} = \det M(\lambda) = \det k_1 \cdot \det k_2$ we obtain the following corollary.
\end{remark}
\begin{corollary}
Set $P_1 = diag(1,1,1,-1,1...,1)\hbox{ and }P_2 = diag(1,1,1,1,-1,1...,1)$

as in Section 2.2.
\begin{enumerate}\item  If $\det \hat{S}=1$, then  we have $M(\lambda)\in \Lambda SO^+(1,n+3)_{\sigma}$ and $P_2P_1k\in K$;
\item  If $\det \hat{S}=-1$, then  we have $M(\lambda)P_1\in \Lambda SO^+(1,n+3)_{\sigma}$ and $P_1k\in K$.
\end{enumerate}
\end{corollary}

With these considerations in mind, we are finally able to determine the transformation formula for the normalized potentials of a Willmore surface under an orientation reversing symmetry.

\begin{theorem}\label{thm-sym-n-F_-}
Let $\D$ denote the unit disk or the complex plane and $y: \D \rightarrow S^{n+2}$ a Willmore surface. Let $z_0$ denote a base point in $\D$.
Let $(\mu , S )$ be
an orientation reversing symmetry of the harmonic Gauss map $f: \D \rightarrow SO^+(1,n+3)/ SO^+(1,3)\times SO(n)$ of the Willmore surface $y$.
Let $F(z,\bar{z},\lambda )$ be the extended frame of $f$ satisfying the initial condition $F(z_0,\bar{z}_0,\lambda)=I$ and
assume that
\[F(z,\bar{z},\lambda )=F_-(z,\lambda )\cdot F_+(z,\bar{z},\lambda ).\]

Then\begin{equation}\label{eq-mu-F_-}
F_-(\mu(z), \lambda) =
 M (\lambda)   \overline{F_-(z,\lambda^{-1})} \cdot W_+(\bar{z},\lambda),
\end{equation}
where
\begin{equation}
W_+( \bar{z}, \lambda) =  \overline{F_+(z,\bar{z}, \lambda^{-1})}  \cdot  k (z, \bar{z}) \cdot (F_+(\mu(z), \overline{\mu(z)},\lambda))^{-1}.
\end{equation}
For the normalized potential $\eta = F_-^{-1} \dd F_-$  of $f$, we obtain from \eqref{eq-mu-F_-}
\begin{equation}\label{eq-mu-eta}
\mu^*\eta(z,\lambda) =  \overline{\eta(z, \lambda^{-1})} \sharp W_+,
\end{equation}
where $``\sharp W_+$'' means  gauging by $W_+$.
Moreover,
\begin{enumerate}\item  if $\det \hat{S}=1$, then $W_+$ takes value in  $ \Lambda^+ SO^+(1,n+3,\C)_{\sigma} $;
\item  if $\det \hat{S}=-1$, then $P_2W_+$ takes value in  $ \Lambda^+ SO^+(1,n+3,\C)_{\sigma}.$
\end{enumerate}
\end{theorem}

\begin{proof}
Since the extended frame $F$ is real, it satisfies
\begin{equation*}
F(z, \bar{z}, \lambda) = \overline{F(z, \bar{z},\lambda)}
\end{equation*}
for all $z\in\D$ and all $\lambda\in S^1$. Then, in view of (\ref{transF}),  we have
\[F(\mu(z), \overline{\mu(z)},\lambda) = M (\lambda) \cdot  F(z, \bar{z}, \lambda^{-1})  \cdot  k (z, \bar{z}) = M (\lambda) \cdot  \overline{F(z, \bar{z}, \lambda^{-1})}  \cdot  k (z, \bar{z}).\]
Therefore, considering the usual Birkhoff splitting (locally near $z=0$)
\[F = F_- F_+, ~\hbox{ with }~~ F_- = I + \mathcal{O}(\lambda^{-1}),\] we obtain
\[F_-(\mu(z), \lambda)\cdot F_+(\mu(z), \overline{\mu(z)},\lambda)=
 M (\lambda)   \overline{F_-(z,\lambda^{-1})}\cdot\overline{F_+(z, \bar{z},\lambda^{-1}) } k (z, \bar{z}).\]
The formulas \eqref{eq-mu-F_-} and \eqref{eq-mu-eta} now follow directly.
\end{proof}

\subsection{From potentials to surfaces}
The converse of the above theorem is

\begin{theorem}\label{thm-sym-F-mu}
Let $\D$ denote the unit disk or the complex plane. Let $\eta$ be a potential for some harmonic map $f:\D\rightarrow SO^+(1,n+3)/ SO^+(1,3)\times SO(n)$,
which is the  oriented  conformal Gauss map of a Willmore surface $y$.
  Let $\mu$ be an anti-holomorphic automorphism of $\D$
and assume that equation \eqref{eq-mu-eta}
holds for some $W_+$ taking values in $  (\Lambda ^+SO^+(1,n+3,\C)_{\sigma})^{\circ}$ $($or $P_2\cdot   (\Lambda ^+SO^+(1,n+3,\C)_{\sigma})^{\circ} )$.
Then  there exists some $ M (\lambda ) \in
(\Lambda SO^+(1,n+3,\C)_{\sigma})^{\circ}$ $($or $ M (\lambda ) \in
P_2\cdot(\Lambda SO^+(1,n+3,\C)_{\sigma})^{\circ} )$ such that for the solution $C$ to
$\dd C = C \eta,\ ~C (z=0,\lambda ) = I$ the following equation is satisfied
\begin{equation}\label{symmfor-C-mu}
 C (\mu(z),\lambda) = M(\lambda) \overline{C(z,\lambda^{-1})} \cdot W_+.
\end{equation}
  Moreover, $\mu$ induces a symmetry of the harmonic map $f$ associated with $\eta$ if and only if $W_+$ and $M$ can be chosen such that
 $M(\lambda) \in (\Lambda SO^+(1,n+3)_{\sigma})^{\circ}$ $($or $ M (\lambda ) \in
P_2\cdot(\Lambda SO^+(1,n+3,\C)_{\sigma})^{\circ} )$ and \eqref{symmfor-C-mu} holds.
In this case $\mu$ induces the symmetry
\begin{equation}\label{symmfor-f-mu}
f(\mu(z), \overline{\mu(z)}, \lambda) = M(\lambda) f(z, \bar{z}, \lambda^{-1})^{\star}.
\end{equation}
of the harmonic map $f$ induced by $\eta$. For the Willmore surface $y$ obtained from $f$ we have furthermore
\begin{equation}\label{eq-y-sym1}
 y(\mu(z), \overline{\mu(z)},\lambda) =   [ M(\lambda) Y(z,\bar{z},\lambda^{-1})].
\end{equation}
\end{theorem}

\begin{proof}
Integrating $\dd C = C \eta$ with $C(z=0,\lambda ) = I$ and using \eqref{eq-mu-eta} one obtains \eqref{symmfor-C-mu}. Consider the Iwasawa decomposition $C=F\cdot V_+$, we have
\[C(\mu(z),\lambda)=F(\mu(z), \overline{\mu(z)},\lambda)\cdot F_+(\mu(z), \overline{\mu(z)},\lambda)=M (\lambda)  \overline{F(z,\bar{z},\lambda^{-1})\cdot F_+(z,\bar{z},\lambda^{-1})}   \cdot W_+.\]
From this we derive
\[F(\mu(z), \overline{\mu(z)},\lambda)=M (\lambda)  \overline{F(z,\bar{z},\lambda^{-1})}\cdot k (z,\bar{z})\]
for some $ k (z,\bar{z})= \overline{F_+(z,\bar{z},\lambda^{-1})} W_+ (F_+(\mu(z), \overline{\mu(z)},\lambda))^{-1}\in\Lambda O^+(1,n+3)_{\sigma}\cap \Lambda^+ O^+(1,n+3,\C)_{\sigma}= O^+(1,3)\times O(n)$. Therefore \eqref{symmfor-f-mu} follows since $\mu$ reverses the orientation of $\D$. The rest of the proof follows from the unique correspondence of Willmore surfaces and their oriented conformal Gauss maps \cite{DoWa1}.
\end{proof}

From the proofs of the theorems above we  obtain the following useful result:
\begin{corollary}\label{cor-non--ori} Retaining the assumptions and the notation of the theorems above we obtain
\begin{enumerate}\item   $f$ has the symmetry \eqref{symmfor-f-mu} for some  $M(\lambda)\in (\Lambda SO^+(1,n+3)_{\sigma})^o$, if and only if $C$ satisfies
\begin{equation}\label{eq-mu-CF}
 (\overline{C(z,\lambda^{-1})})^{-1}\cdot  M(\lambda)^{-1} \cdot C(\mu(z),\lambda)\in \Lambda^+ SO^+(1,n+3,\C)_{\sigma}.
\end{equation}
\item   $f$ has the symmetry \eqref{symmfor-f-mu} for some  $M(\lambda)\in P_2(\Lambda SO^+(1,n+3)_{\sigma})^o$, if and only if $C$ satisfies
\begin{equation}\label{eq-mu-CF2}
 (\overline{C(z,\lambda^{-1})})^{-1}\cdot  M(\lambda)^{-1} \cdot C(\mu(z),\lambda)\in P_2\Lambda^+ SO^+(1,n+3,\C)_{\sigma}.
\end{equation}
\end{enumerate}
\end{corollary}
\begin{proof}  From \eqref{symmfor-C-mu}, we see that $f$ has the symmetry \eqref{symmfor-f-mu} if and only if $C$ satisfies
\[(\overline{C(z,\lambda^{-1})})^{-1}\cdot M(\lambda)^{-1} \cdot\mu^*C(z,\lambda)=W_+.\]
Then \eqref{eq-mu-CF} and \eqref{eq-mu-CF2} follow.
\end{proof}

  \begin{remark}\
\begin{enumerate}
\item  Although the above formulas seem somewhat complicated, they actually do have many useful applications which will be shown in the following sections. A particularly  interesting example is the description of isotropic Willmore immersions from $\R P^2$ in $S^4$ which we will discuss in Section 5.
\item As we have pointed out for several times, for the $S^2$ case, the idea is to to proceed as usual on the two charts $\D=S^2\setminus\{z_0\}$ with $z_0=0$ or $z_0=\infty$ (See Section 4 of  \cite{DoWa1}). Then the loop group theory works well on $\D$ and the equations concerning symmetries discussed above hold on $\D\setminus\{\mu(z_0)\}$.
\end{enumerate}
\end{remark}

\section{Non--orientable Willmore surfaces as quotients of orientable Willmore surfaces with symmetries}

This section aims to apply the results above to the discussion of non--orientable Willmore surfaces. For this purpose, we will first recall some well--known facts about anti--holomorphic automorphisms of orientable surfaces. Then, we apply Theorem \ref{thm-sym-F-mu} of Section 3 to non--orientable Willmore surfaces by  viewing them as quotients of orientable Willmore surfaces by  orientation reversing
involutions $(\mu, I)$.

\subsection{Non-orientable surfaces with universal cover $S^2$}

For non-orientable surfaces with universal cover $S^2$, we have  the following well-known result (See e.g. \cite{Eujalance} for a reference)
\begin{theorem} \label{RP2}
Let $M$ be a non--orientable surface  obtained from its universal cover $\tilde{M}$
by an anti-holomorphic map $\mu$ of finite order which generates a freely acting group.  Then $\tilde M= S^2$,
 $M= \R P ^2$ and w.l.g. $\mu(z) = -\frac{1}{\bar{z}}$.
 In particular, $\mu$ has order two and no fixed points.
Moreover, $\R P^2$ is the only non-orientable surface with universal cover $S^2.$
\end{theorem}


\subsection{Non--orientable Willmore surfaces}

In this subsection we will discuss how one can construct Willmore immersions from non--orientable surfaces to $S^{n+2}$.
Let $\check{M}$ be a (not necessarily simply-connected) Riemann surface with an anti-holomorphic automorphism $\mu:\check{M}\rightarrow \check{M}$ of order $2$ without any fixed points. Then $M=\check{M}/\{p\sim\mu(p)\}$ is a non--orientable surface with the natural $2 :1$ covering $\pi:\check{M}\rightarrow M$, $p\rightarrow \{p\sim \mu(p)\}$. Moreover, all non--orientable surfaces can be obtained this way.

\subsubsection{Lifting and descending Willmore surfaces}

Let $M$ and $\check{M}$ be as above and let $y: M \rightarrow S^{n+2}$ be a  Willmore surface (with or without branch points).  Now let $\check{y}: \check{M} \rightarrow S^{n+2}$ be  the natural  lift of $y$ satisfying $\check{y}(p)=y(\pi(p))$ for any $p\in\check{M}$. Let $\check{f}:\check{M}\rightarrow S^{n+2}$ denote the conformal Gauss map of $\check{y}$.
Then we obtain \[\check{f}\circ\mu=\check{f}^{\star},\] since $\mu$ reverses the orientation (see also Theorem \ref{th-symm-f}).
As a consequence, the Willmore immersion $\check{y}$ and its conformal  Gauss map $\check{f}$   inherit a natural symmetry
\begin{equation}\label{eq-non--ori-sym}
\check{y}(\mu(p))=\check{y}(p),\hspace{5mm}  \check{f}(\mu(p))=\check{f}(p)^{\star}.
\end{equation}
 Conversely, any Willmore immersion $\check{y}: \check{M} \rightarrow S^{n+2}$
having such a pair of symmetries  can be factored through the non--orientable Willmore surface $M$.
For later purposes we would like to point out that the anti-holomorphic transformation $\mu$
yields,  in our notation,  the special form symmetry  $(\mu, I)$.  Therefore,  when applying results of the previous sections, we will thus have the case, where $\det \hat{S} = 1.$

Since we are interested in the construction of   Willmore surfaces (with or without  branch points,  from non--orientable surfaces to $S^{n+2}$, we outline below in some detail how this can be achieved in the loop group formalism.
Note that from what was said above, in the case $\check{M} = S^2$
we have $\check{M} = \tilde{M} = S^2$ and only need to consider one specific $\mu$. We will consider this case in the next section in more detail separately.

 Recall, if we  consider the Riemann surface $\check{M}$ and a Willmore surface
$\check{y}: \check{M} \rightarrow S^{n+2}$, then in Section 2, we have constructed the conformal Gauss map $\check{f},$ an extended frame $\check{F}$, a holomorphic/meromorphic
frame $\check{C}$ and the potential $\check{\eta} = \check C^{-1} \dd \check C$.
We also have normalized $\check{F}$ and $\check{C}$ so that they attain the value $I$ at some base point. In \cite{DoWa-sym1} we have stated the transformation behavior of $\check{f}$, $\check{F}$, $\check{C}$ and $\check{\eta}$ under the action of $\Gamma = \pi_1 (\check{M})$
on the universal cover $\tilde{M}$ of $\check{M}$.

Assume now that $\mu$ is an anti--holomorphic fixed point free involution of $\check M$ satisfying \eqref{eq-non--ori-sym}.
Since $(\mu, I)$ acts as a symmetry of $\check y$ on $\check{M}$,
 its action on $\tilde{M}$
induces transformation rules for the lifts $\tilde y$, $\tilde{f}$, $\tilde{F}$, $\tilde{C}$ and $\tilde{\eta}$. One should note that while the action of $\tilde{\mu}$ on
$\tilde{y}$ is actually trivial, its action on the associated family of
$\tilde{M}$ is in general non-trivial.

 Using what was just recalled,  one can construct  now what we want in two steps:
\subsubsection{ Construction procedure of Willmore surfaces from non-orientable surfaces to $S^{n+2}$}

 In the first step we consider $\Gamma_0 \cong\pi_1 (\check{M})$,  as subgroup of the  group of holomorphic  automorphisms  $\tilde M$.  Then $\pi_1 (M) =
\Gamma \cup \mu \Gamma$, where $\mu$ can be considered as a fixed point free anti-holomorphic transformation on $\tilde{M}$ satisfying $\mu^2 \in \Gamma_0$. Let's consider a potential which generates a Willmore immersion $\check{y}: \check{M} \rightarrow S^{n+2}$.
We have seen in \cite{DoWa-sym1} that for the construction of a Willmore immersion $\check{y}: \check{M} \rightarrow S^{n+2}$ we could start from  an invariant potential $\eta$ on $\tilde{M}$. Then the solution $C$ to $\dd C = C \eta$
with $C(0,\lambda) = I$ satisfies $g^*C = \chi(g, \lambda) C$ for all $g \in \Gamma$. We need to make sure that  $\chi (g,\lambda) \in \Lambda SO^+(1, n+3)_\sigma^{\circ}. $ Moreover we need to make sure that $\chi (g ,\lambda = 1) = I$ for all $g \in \Gamma_0$.
Then an Iwasawa splitting $C = F F_+$ with $F(0,\lambda) = I$
yields a conformally harmonic map $\check{f} : \check{M} \rightarrow S^{n+2}$
by putting $\check{f} \equiv F \mod K$.

 To actually obtain a  family $\check{y}_\lambda$ of (possibly) branched surfaces
from $\check{M}$ to $S^{n+2}$ we finally need to make sure, as discussed in Theorem 3.11 of \cite{DoWa1}, that  $\check{f}$ can be realized as the conformal Gauss map of
$\check{y}_\lambda$.
If all this works, then
so far we have constructed  the Willmore surface
$\check{y}: \check{M} \rightarrow S^{n+2}$.
This finishes step 1.

In step 2 we incorporate the action of $\mu$.  In general this works as follows:
From Section 3 we know that the potential $\eta$ needs to satisfy the relation
\begin{equation}\label{eq-non--ori-potential}
\mu^* \eta(z,\lambda) =  \overline{\eta(z, \lambda^{-1})}  \sharp \widetilde{W}_+,
\end{equation}
for  $\widetilde{W}_+\in\Lambda^+SO^+(1,n+3,\C)_{\sigma}$, where $``\sharp \widetilde{W}_+$'' means  gauging by $\tilde{W}_+$.
If this is satisfied, then one obtains
\begin{equation} \label{eq-non--ori-C}
C(\mu(z),\lambda) = \chi ( \lambda) \overline{C(z,\lambda^{-1})}   \widetilde{W}_+(\bar{z},\lambda)
\end{equation}
and also
\begin{equation} \label{eq-non--ori-F}
F(\mu(z), \overline{\mu(z)},\lambda) =  \chi (\lambda) F(z,\bar z,\lambda^{-1})  \cdot k(z, \bar{z}).
\end{equation}
Note that we use the notation $\chi(\lambda)$ for the monodromy of $\mu$ and write $\chi(g,\lambda)$ for the monodromy of $g \in \Gamma_0.$
In order to obtain a (possibly branched) Willmore immersion of $M=\check{M}/\{p\sim\mu(p)\}$, two more conditions need to be satisfied: on the one hand
we need to require
$ \chi (  \lambda) \in \Lambda SO^+(1, n+3)_\sigma^{\circ} $
and on the other hand we need  $\chi(\lambda)_{ \lambda = 1} = I$.
 Altogether we obtain this way that $(\mu,\chi)$ is an orientation reversing  symmetry
of $\check{y}_\lambda$ which  is the symmetry $(\mu, I)$ for $\lambda = 1$ and therefore produces a Willmore surface $y: M \rightarrow S^{n+2}.$

Altogether we obtain as an application of Theorem \ref{thm-sym-F-mu}

\begin{theorem}\label{thm-sym-n-ori} We retain the notation introduced above.
\
\begin{enumerate}
\item If $\check{M}\neq S^2$, we start from some $\pi_1(\check{M})$ invariant   potential $\eta$ satisfying
\eqref{eq-non--ori-potential} for some $\widetilde{W}_+$. Then the solution to the ode
\[\dd C = C \eta,\ C(0,\lambda) = I\] satisfies \eqref{eq-non--ori-C}.
Let $(\mu,\chi)$ be a symmetry as above. Assume that  $ {\chi(  \lambda)} \in \Lambda SO^+(1,n+3)_{\sigma}$ for all $\lambda \in S^1$ and
$ {\chi(  \lambda)}|_{\lambda=1}=I$.
Then we obtain a harmonic map $\check{f}|_{\lambda=1}$ defined on $\check{M}$. Moreover, the corresponding frame $\check{F}$, obtained from $C$ by the unique Iwasawa splitting, also satisfies \eqref{eq-non--ori-F}.
From this we obtain
\begin{equation}\label{eq-fcheck}
\check{f}(\mu(z), \overline{\mu(z)},\lambda)
=  {\chi (  \lambda)}\check{f}(z,\bar{z}, \lambda^{-1})^{\star}
\end{equation}
and, if $\check{f}$ is the oriented conformal Gauss map of  a  Willmore immersion $\check{y}$, then we also have
\begin{equation}\label{eq-ycheck}
\check{y}(\mu(z), \overline{\mu(z)},\lambda) = \mu{ \chi(\mu,  \lambda)}\check{y}(z,\bar{z},
\lambda^{-1})=[ {\chi(  \lambda)}\check{Y}(z,\bar{z},\lambda^{-1})].
\end{equation}
All these together show that $\mu$ leaves  $\check{y}(z,\bar{z},\lambda = 1)$ invariant. Therefore  one obtains a non--orientable Willmore surface $y=\check{y}(p)=\check{y}(\mu(p))$ on $M=\check{M}/\{p\sim\pi(p)\}$.
\item If $\check{M}=S^2$, we start from some normalized potential $\eta$ satisfying
\eqref{eq-non--ori-potential} defined on $\C\subset S^2$ for some $\widetilde{W}_+$. Note that in this case we have $\mu(0)=\infty$. Then the rest is the same as in $(1)$,  except that the results hold on $\C\setminus\{0\}$. Setting $\lambda=1$ and taking the limit $z\rightarrow0$ for \eqref{eq-fcheck} and \eqref{eq-ycheck}, we also obtain the definition of $\check{f}|_{\lambda=1}$ and $\check{y}|_{\lambda=1}$ at  $\infty=S^2\backslash\C$.
\end{enumerate}
\end{theorem}
\begin{proof}
The only thing left to show is $(2)$. As we have discussed in Section \ref{section-s2}, the extended frame is well-defined on $\C$ and the potential is meromorphically defined on $S^2$ (See \cite{DoWa1} more details for the potentials on $S^2$). And all the results on the symmetry of $\mu$ can be discussed
 on $\C\setminus\{0\}$ without any further changes. Then $(2)$ follows.
\end{proof}
In the above theorem we start from a potential satisfying \eqref{eq-non--ori-potential}. This involves the matrix $\widetilde{W}_+$ and hence  it is a complicated condition that the potential needs to satisfy. In the case of orientable Willmore surfaces (compact or non-compact), we can show that one can avoid a term like  $\widetilde{W}_+$ by the right choice of potential.  Unfortunately, for non--orientable Willmore surfaces this is not the case.
We present this result only for $S^2$, but expect that such a result holds more generally.

\begin{theorem}
Let $y:S^2\rightarrow S^{n+2}$ be a Willmore immersion satisfying $y\circ\mu=y$ for $\mu(z)=-\frac{1}{\bar{z}}$.  Then there exists no potential $\eta$ for $y$ such that
$\mu^*\eta=P\overline{\eta(z,\lambda^{-1})}P^{-1}$ holds for some $P\in SO^+(1,3)\times SO(n)$.
\end{theorem}
\begin{proof} We can assume that the extended frame of $y$ and the solution $C$ to $\dd C=C\eta$ all attain the value $I$ at the base point $z_0=0$.
Let $\eta_-$ denote the normalized potential associated with $y$ and $z_0=0$. then the meromorphic extended frame $C_-$ associated with $\eta_-$ and $C$ are in the relation
$$C=C_-\cdot h_+,\ \hbox{ with } h_+: S^2\rightarrow \Lambda^+G^{\C}_{\sigma}.$$
Then we obtain
\[C(\mu(z),\lambda)=\chi(\lambda)\overline{C(z,\lambda^{-1})}
=\chi(\lambda)\overline{C_-(z,\lambda^{-1})}W_+(\bar{z},\lambda) h_+(\mu(z),\lambda),\]
 with
 \[W_+(\bar{z},\lambda)=\overline{h_+(z,\lambda^{-1})}\cdot h_+(\mu(z),\lambda)^{-1}\]
 on $S^2\setminus\{0,\infty\}$.
This can be rewritten in the form
\begin{equation}\label{eq-W+-Bir-Decom}
    W_+(\bar{z},\lambda)=A_+(\bar{z},\lambda) h_+(-\frac{1}{\bar{z}},\lambda).
\end{equation}
Note, by our assumptions $C$ and $C_-$ are finite at $z_0=0$, whence also $h_+$ and $A_+$ are finite at $z_0=0$. On the other hand, $W_+(\bar{z},\lambda)$ is not necessarily finite at $z_0=0$.
We observe that \eqref{eq-W+-Bir-Decom} represents a Birkhoff decomposition of $W_+(\bar{z},\lambda)$, considered as a function of $\bar{z}$ with parameter $\lambda$. In particular, $W_+$ is in the big Birkhoff cell relative to $\bar{z}$. Since $W_+$, $A_+$ and $h_+$ all permit the continuous limit $\lambda\rightarrow0$,
\begin{equation}\label{eq-W+-Bir-Decom-2}
    W_0(\bar{z})=A_0(\bar{z}) h_0(-\frac{1}{\bar{z}}).
\end{equation}
On the other hand, the equation
\begin{equation}\label{eq-C_--W_+-1}
    C_-(\mu(z),\lambda)=\chi(\lambda) \overline{C_-(z,\lambda^{-1})}W_+(\bar{z},\lambda).
\end{equation}
can be read as
\begin{equation}\label{eq-C_--W_+-2}
  W_+(\bar{z},\lambda)=\left(\chi(\lambda)\overline{C_-(z,\lambda^{-1})}\right)^{-1} C_-(-\frac{1}{\bar{z}},\lambda).
\end{equation}
Obviously, also this product represents a Birkhoff decomposition of $W_+$ considered as a function of
$\bar{z}$ with parameter $\lambda$. Since the limit $\lambda\rightarrow 0$ exists for $W_+$ we have two possibilities:

\begin{enumerate}\item The limit stays in the big Birkhoff cell relative to $\bar{z}$. In this case the factors need to converge separately. But $C_-$ has no limit when $\lambda\rightarrow0$ since $C_-(\lambda)$ is at least a polynomial in $\lambda^{-1}$ and non-constant.

 \item The limit does not stay in the big Birkhoff cell relative to $\bar{z}$. Then $W_0=\mathcal{U}^+\delta \mathcal{V}^{-}$ with $\mathcal{U}^+$ anti-holomorphic at $z=0$, $\mathcal{V}^{-}$ anti-holomorphic at $z=\infty$ and $\delta=\delta(z)\neq I$ (a Weyl group element). This contradicts  \eqref{eq-W+-Bir-Decom-2} and proves the claim.
     \end{enumerate}
\end{proof}

\begin{remark}\
From the above it is clear that the construction of non--orientable Willmore surfaces is more complicated than the construction of orientable ones. Indeed, we have found only a few papers dealing with this issue.
 In the next section  we will give some examples for Willmore surfaces of type $\R P^2$.
\end{remark}

\section{Isotropic Willmore immersions from ${\R P}^2$ into $S^{4}$}

With the notation  of Section 4.1, let  us now consider the case $ \tilde{M} = \check{M}  = S^2$. We still assume that the group of deck transformations corresponding to
$\pi_1 (M)$ contains an anti-holomorphic transformation which, since it is fixed point free, is automatically of order 2. By Theorem \ref{RP2} we know that this means that we consider
the case $ \tilde{M} = \check{M}  = S^2$ and $M = \R P ^2$, and that we only need to deal with the anti-holomorphic automorphism
\begin{equation}\label{eq-s2-mu}
     \mu:\tilde{M} = S^2\rightarrow S^2,\ \mu(z)=-\frac{1}{\bar{z}}.
\end{equation}

 To illustrate the theory presented in Section 4.1  and to find new examples of (possibly branched) Willmore  immersions of $\R P^2$ into $S^{n+2}$, we will consider first all isotropic Willmore two spheres in $S^4$  and then we will single out all those among them which descend to  Willmore surfaces from $\R P^2$ to $S^4$. We choose these surfaces mainly due to  two reasons. Firstly, it is easy to write down all normalized potentials for isotropic Willmore surfaces in $S^4$, see below or \cite{DoWa1}. Secondly, although since the classical work of Bryant \cite{Bryant1982} there have been many descriptions of isotropic surfaces in $S^4$,   examples of non--orientable Willmore surfaces are still rare. So far, in the literature we only find discussions on this topic in \cite{Ejiri-equ} and \cite{Ishihara}.  For instance, there are very few explicit examples of minimal  $\R P^2$ into $S^4$ except {the Veronese} surface. So it is worthwhile for an illustration of our work to consider isotropic Willmore surfaces in $S^4$.

Let $y:  S^2\rightarrow S^4$ be an isotropic Willmore surface.  Then its normalized potential has the form \cite{DoWa1}
\[\eta=\lambda^{-1}\left(
                     \begin{array}{cc}
                       0 & \hat{B}_1\\
                      -\hat{B_1}^tI_{1,3} & 0 \\
                     \end{array}
                   \right)\dd z,\]
with
\begin{equation}\label{eq-b1}
\hat{B}_1=\frac{1}{2}\left(
                     \begin{array}{cccc}
                      i(f_3'-f_2')&  -(f_3'-f_2') \\
                      i(f_3'+f_2')&  -(f_3'+f_2')   \\
                      f_4'-f_1' & i(f_4'-f_1')   \\
                      i(f_4'+f_1') & -(f_4'+f_1')  \\
                     \end{array}
                   \right),\ \ \ f_1'f_4'+f_2'f_3'=0.
                   \end{equation}
 Here the functions $f_j,$ $1\leq j\leq4$, are meromorphic functions on $S^2$.
\begin{theorem}\label{thm-rp2} Let $y:  S^2\rightarrow S^4$ be an isotropic Willmore surface with its normalized potential as above.
\begin{enumerate}
\item Assume that $y$ has the symmetry
$y\circ\mu=y$. Then the following equations hold
 \begin{equation}\label{eq-iso-rp2}
    \left\{
    \begin{split}
&f_1(z)+(f_1(z)f_4(z)+f_2(z)f_3(z))\overline{f_4(\mu(z))}=0,\\
&f_2(z)+(f_1(z)f_4(z)+f_2(z)f_3(z))\overline{f_2(\mu(z))}=0,\\
&f_3(z)+(f_1(z)f_4(z)+f_2(z)f_3(z))\overline{f_3(\mu(z))}=0,\\
&f_4(z)+(f_1(z)f_4(z)+f_2(z)f_3(z))\overline{f_1(\mu(z))}=0.\\
\end{split}\right.
\end{equation}
\item Conversely,  consider a normalized potential $\eta$ of the form  \eqref{eq-b1} with meromorphic functions $f_j$. Assume that $f_j$ satisfies
  \eqref{eq-iso-rp2} and that $\eta$ has a meromorphic antiderivative.
 Then the harmonic map $f|_{\lambda=1}$ induced by $\eta$ can be projected to a map $y$ from $S^2$ to $S^{n+2}$.
 Moreover, we have
 \begin{enumerate}
\item
If $y(z,\bar{z},\lambda)$ is not a constant map, then we have
\[y(\mu(z),\overline{\mu(z)},\lambda)=y(z,\bar{z},\lambda),\] and  the harmonic map
$f|_{\lambda=1}$  is the oriented conformal Gauss map of the Willmore surface $y|_{\lambda=1}$.

\item Moreover, $y$ is conformally congruent to some minimal $\R P^2$ if and only if there exists some none-zero time-like vector $\mathrm v\in\R^4_1$ such that $\mathrm v^tI_{1,3}B_1=0$. In particular, $y$ is conformally congruent to some minimal $\R P^2$ if $f_2=f_3$.
\end{enumerate}
\end{enumerate}
\end{theorem}

Note that (2-b) of Theorem \ref{thm-rp2} is a corollary to Theorem 1.3 of \cite{Wa4}, and by (2-b) of Theorem \ref{thm-rp2} we obtain all minimal $\R P^2$ in $S^4$.
We will illustrate the theorem by examples and leave the technical proof of (1) and (2-a) to the appendix.

A class of special solutions to \eqref{eq-iso-rp2} is the following
\begin{lemma}\label{lemma-f_ij}
Retaining the notation and the assumptions of Theorem \ref{thm-rp2}, let
\begin{equation}\label{eq-f_ij}
  f_1=-2mz^{2m+1},\ f_2= f_3=i\sqrt{4m^2-1}z^{2m},\ f_4=-2mz^{2m-1},
\end{equation}
with $m\in \mathbb{Z}^+$.
Then the functions $f_j, j=1,2,3,4,$ satisfiy  \eqref{eq-iso-rp2}.
\end{lemma}

Using this lemma, it is easy to write down infinitely many potentials which will lead to different Willmore immersions from $\R P ^2$ into $S^4$.  This family contains as a special case  the well--known Veronese sphere  in $S^4$ \cite{Wang-S4}. Moreover, from \cite{Wa4} one can see  that all these Willmore surfaces are conformally  equivalent to some minimal surfaces in $S^4$.
First we mention the well--known
\begin{example}{\bf Veronese sphere in $S^4$.}\ ~  It is well known (see for example Example 5.21 of \cite{DoWa1}) that  the map
$y: S^2 \rightarrow S^4$, given by the formula
\begin{equation}
y=\left(\frac{2r^2-(1-r^2)^2}{(1+r^2)^2},
\frac{\sqrt{3}(z+\bar{z})(1-r^2)}{(1+r^2)^2},\frac{-i\sqrt{3}(z-\bar{z})(1-r^2)}{(1+r^2)^2},\frac{\sqrt{3}(z^2+\bar{z}^2)}{(1+r^2)^2},\frac{-i\sqrt{3}(z^2-\bar{z}^2)}{(1+r^2)^2}\right),
\end{equation}
defines  a Willmore immersion into $S^2$ which is invariant under the action of $\mu$   given in (\ref{eq-s2-mu}) above.
Moreover, a lift $Y(z,\bar{z}, \lambda)$ of the associated family of $y$ is given by (\cite{DoWa1}, \cite{Wang-S4})
\begin{equation}
\begin{split}
Y(z,\bar{z},\lambda)=&~\frac{1}{2}\left(\frac{1+r^2}{\sqrt{3}}, \frac{2r^2-(1-r^2)^2}{\sqrt{3}(1+r^2) },
\frac{(z+\bar{z})(1-r^2)}{1+r^2 },\frac{-i(z-\bar{z})(1-r^2)}{1+r^2},
\right.\\
&~~~~~~~~~~~~~\left.\frac{\lambda^{-1}z^2+\lambda\bar{z}^2}{1+r^2},\frac{-i(\lambda^{-1}z^2-\lambda\bar{z}^2)}{1+r^2}\right).\\
\end{split}
\end{equation}
From this formula it is easy to verify that $y$ actually is an
``$S^1-$invariant" finite uniton type immersion in the sense of \cite{BuGu}, since one observes
\begin{equation}
Y(z, \bar{z},\lambda) = R_{\lambda} Y(z,\bar{z},\lambda=1),
\end{equation}
where
\[R_{\lambda}=\left(
                \begin{array}{cc}
                  I_4 & 0 \\
                  0 & \tilde{R}_{\lambda} \\
                \end{array}
              \right), \  \tilde{R}_{\lambda}=\left(
                           \begin{array}{cc}
                             \frac{\lambda^{-1}+\lambda}{2}&  \frac{i(\lambda^{-1}-\lambda)}{2} \\
                              \frac{-i(\lambda^{-1}-\lambda)}{2} &  \frac{\lambda^{-1}+\lambda}{2}\\
                           \end{array}
                         \right).\]
One can check easily
\[Y(\mu(z),\overline{\mu(z)},\lambda)=\frac{1}{r^2}R_{\lambda}^2\cdot Y(z,\bar{z},\lambda^{-1})=\frac{1}{r^2}Y(z,\bar{z},\lambda).\]
That is,
\[y(\mu(z),\overline{\mu(z)},\lambda)=[Y(\mu(z),\overline{\mu(z)},\lambda)]=[Y(z,\bar{z},\lambda)]=y(z,\bar{z},\lambda).\]

Moreover, setting $m=1$ in \eqref{eq-f_ij}, we obtain $ f_1=-2z^3,\ f_2= f_3=i\sqrt{3}z^2,\ f_4=-2z$. By Formula (2.7) of \cite{DoWa-sym1} (see also \cite{Wang-S4}),
the corresponding Willmore surfaces form exactly the associate family of the Veronese sphere with
$Area(y)=W(y)+4\pi=12\pi.$
\end{example}

\begin{example} Setting $m=2$ in \eqref{eq-f_ij}, we have $f_1=-4z^5,\ f_2= f_3=i\sqrt{15}z^4,\ f_4=-4z^3.$ So the matrix $\hat B_1$ in its normalized potential is of the form
\[ \hat{B}_1= \left(
                     \begin{array}{cccc}
                      0&  0  \\
                      -4\sqrt{15}z^3 & -4i\sqrt{15}z^3  \\
                      -2z^2(3-5z^2) & -2iz^2(3-5z^2)     \\
                      -2iz^2(3+5z^2) & 2z^2(3+5z^2)     \\
                     \end{array}
                   \right). \]
Let $y$ be the corresponding Willmore surface. By formula (2.7) of \cite{DoWa-sym1} (see also \cite{Wang-S4}),
a lift $Y(z,\bar{z}, \lambda)$ of the associated family of $y$ is given by
\begin{equation}\label{eq-non--ori-example2}
Y(z,\bar{z},\lambda)=\left(
                       \begin{array}{c}
                         (3r^4-4r^2+3)(1+r^2)^2 \\
                         -(3r^8-8r^6+8r^4-8r^2+3) \\
                         \sqrt{15}(z+\bar{z})(1-r^2)(1+r^4) \\
                        -i\sqrt{15}(z-\bar{z})(1-r^2)(1+r^4) \\
                        \sqrt{15}(\lambda^{-1}z^4+\lambda\bar{z}^4) \\
                         i\sqrt{15}(\lambda^{-1}z^4-\lambda\bar{z}^4) \\
                       \end{array}
                     \right)=\left(
                               \begin{array}{c}
                                 y_0 \\
                                 y_1 \\
                                 y_2 \\
                                 y_3 \\
                                 y_4 \\
y_5\\
                               \end{array}
                             \right)
.
\end{equation}
We have
 \begin{equation*}
 \begin{split}
 \langle Y_z (z, \bar z, \lambda), Y_{\bar{z}}(z, \bar z, \lambda)\rangle=30\left(r^4(r^2-2)^2+2r^4+(2r^2-1)^2\right)(1+r^2)^2,\\
\end{split}
\end{equation*}
showing that $y(z, \bar{z},\lambda)$ is an immersion at all $z\in\C$.
Together with
\[y(\mu(z),\overline{\mu(z)},\lambda)=[Y(\mu(z),\overline{\mu(z)},\lambda)]=\left[\frac{1}{r^8}Y(z,\bar{z},\lambda)\right]=y(z,\bar{z},\lambda)\] for $\mu(z)=-\frac{1}{\bar{z}}$ and $z\in\C$, we see that $y$ is a Willmore immersion (without branched points) from  $\R P^2$ to $S^4$. Moreover,
$y=\frac{1}{y_0}\left(y_1,y_2,y_3,y_4,y_5\right)^t$ is in fact a minimal immersion from $\R P^2$ to $S^4$ (see \cite{Wa4}) , with induced metric
\[|y_z|^2=\frac{30(1-4r^2+10r^4-4r^6+r^8)}{(3r^4-4r^2+3)^2(1+r^2)^2}\]
and \[Area(y)=W(y)+4\pi=20\pi.\]
\end{example}

\begin{example} Setting $m=3$ in \eqref{eq-f_ij}, we have $ f_1=-6z^{7},\ f_2= f_3=i\sqrt{35}z^{6},\ f_4=-6z^{5}.$ So its normalized potential is of the form
$$ \hat{B}_1= \left(
                     \begin{array}{cccc}
                      0&  0  \\
                      -6\sqrt{35}z^5 & -6i\sqrt{35}z^5  \\
                      -3z^4(5-7z^2) & -3iz^2(5-7z^2)     \\
                      -3iz^4(5+7z^2) & 3z^2(5+7z^2)     \\
                     \end{array}
                   \right). $$
Let $y$ be the corresponding Willmore surface. By formula (2.7) of \cite{DoWa-sym1} (see also \cite{Wang-S4}), a lift $Y(z,\bar{z}, \lambda)$ of the associated family of $y$ is given by
\begin{equation}\label{eq-non--ori-example3}
Y(z,\bar{z},\lambda)=\left(
                       \begin{array}{c}
                         5+7r^2+7r^{12}+5r^{14} \\
                         -5+7r^2+7r^{12}-5r^{14} \\
                         \sqrt{35}(z+\bar{z})(1-r^{12}) \\
                        -i\sqrt{35}(z-\bar{z})(1-r^{12}) \\
                        \sqrt{35}(\lambda^{-1}z^6+\lambda\bar{z}^6)(1+r^2) \\
                         i\sqrt{35}(\lambda^{-1}z^6-\lambda\bar{z}^6)(1+r^2) \\
                       \end{array}
                     \right)=\left(
                               \begin{array}{c}
                                 y_0 \\
                                 y_1 \\
                                 y_2 \\
                                 y_3 \\
                                 y_4 \\
y_5\\
                               \end{array}
                             \right)
.
\end{equation}
We have
 \begin{equation*}
 \begin{split}
 \langle Y_z (z, \bar z, \lambda), Y_{\bar{z}}(z, \bar z, \lambda)\rangle&=70(1+36r^{10}+70r^{12}+36r^{14}+r^{24})
\end{split} \end{equation*}
showing that $Y(\mu(z),\overline{\mu(z)},\lambda)$ is an immersion at all $z\in\C$.
Together with
\[y(\mu(z),\overline{\mu(z)},\lambda)=[Y(\mu(z),\overline{\mu(z)},\lambda)]=\left[\frac{1}{r^{14}}Y(z,\bar{z},\lambda)\right]=y(z,\bar{z},\lambda)\]
for $\mu(z)=-\frac{1}{\bar{z}}$ and $z\in\C$, we see that $y$ is a Willmore immersion (without branch points) from  $\R P^2$ to $S^4$. Moreover, $y=\frac{1}{y_0}\left(y_1,y_2,y_3,y_4,y_5\right)^t$ is in fact a minimal immersion from $\R P^2$ to $S^4$ (see \cite{Wa4,Wang-S4}), with induced metric
\[|y_z|^2=\frac{70\left( 1+36r^{10}+70r^{12}+36r^{14}+r^{24} \right)}{y_0^2},\]
and
\[Area(y)=W(y)+4\pi=28\pi.\]
\end{example}

\begin{remark}\
\begin{enumerate}
 \item In \cite{Ishihara}, using  a formula of Bryant \cite{Bryant1982} constructing all   isotropic minimal surfaces in $S^4$, Ishihara   provided a description of non--orientable, isotropic harmonic maps into $S^4$  but the immersion property is not discussed \cite{Ishihara}.
It is difficult to compute an immersion explicitly following his approach. Therefore, to the authors' best knowledge, our examples  above are, with the exception of the Veronese spheres, the first explicitly known non--orientable minimal immersions from $\R P^2$ into $S^4.$ Actually, one can produce many more  examples using Lemma \ref{lemma-f_ij}, Theorem \ref{thm-rp2} and  formula (2.7) of \cite{DoWa-sym1} (see also \cite{Wang-S4}).

\item It is straightforward to see that the above examples are all equivariant in the sense of formula (3.1) of \cite{Ejiri-equ}. Moreover, the Veronese sphere is homogeneous and the other two examples given above are not homogeneous.

\item  Using  \cite{DoWa-sym1} and \cite{Wang-S4}, one can show that  Lemma \ref{lemma-f_ij} gives in fact all equivariant isotropic Willmore immersions from $\R P^2$ to $S^4$. As a consequence, all equivariant isotropic Willmore immersions from $\R P^2$ to $S^4$ are conformally congruent to some minimal $\R P^2$ in $S^4$. But if one allows the maps to have branch points, there will be new examples \cite{Wang-S4}.
\item It is  unknown,  whether all  isotropic Willmore immersions from $\R P^2$ to $S^4$ are conformally congruent to some minimal $\R P^2$ in $S^4$ or not.
\end{enumerate}
\end{remark}


\section{Appendix A: Proof of Theorem \ref{thm-rp2}}

\begin{proof} By Theorem 5.6 of \cite{DoWa1}, we see that the normalized potential of $y$ is of the form stated in Theorem \ref{thm-rp2}. Let $f$ be the conformal Gauss map of $y$. So
\[f(\mu(z),\overline{\mu(z)},\lambda)=M(\lambda)f(z,\bar{z},\lambda^{-1})^{\star}.\]
Let $\eta$ be the normalized potential of $f$ and
 let $C$ be a solution to $\dd C=C\eta,$ $C(0,\lambda)=I$.
By Corollary \ref{cor-non--ori}, $f(\mu(z),\overline{\mu(z)},\lambda)=M(\lambda)f(z,\bar{z},\lambda^{-1})^{\star}$ if and only if \eqref{eq-mu-CF} holds, i.e.
$$ (\overline{C(z,\lambda^{-1})})^{-1}\cdot M(\lambda)^{-1} \cdot C(\mu(z),\lambda)\in \Lambda^+ SO^+(1,n+3,\C)_{\sigma}.$$
Hence it suffices to show in this case that \eqref{eq-mu-CF} is equivalent to \eqref{eq-iso-rp2}.

Let $F=C\cdot C_+$ be the extended frame of $y$, then we obtain by \cite{Wang-S4} that
\[F(z,\bar{z},\lambda)=R_{\lambda}F(z,\bar{z}, 1)R_{\lambda}^{-1},\]
where
\[R_{\lambda}=\left(
                \begin{array}{cc}
                  I_4 & 0 \\
                  0 & \tilde{R}_{\lambda} \\
                \end{array}
              \right), \  \tilde{R}_{\lambda}=\left(
                           \begin{array}{cc}
                             \frac{\lambda^{-1}+\lambda}{2}&  \frac{i(\lambda^{-1}-\lambda)}{2} \\
                              \frac{-i(\lambda^{-1}-\lambda)}{2} &  \frac{\lambda^{-1}+\lambda}{2}\\
                           \end{array}
                         \right).
\]
From this we infer that
\[\begin{split}F(\mu(z),\overline{\mu(z)},\lambda)&=R_{\lambda}F(\mu(z),\overline{\mu(z)},1)R_{\lambda}^{-1}\\
&=R_{\lambda} F(z,\bar{z},1)k(z,\bar{z})R_{\lambda}^{-1}\\
&=R_{\lambda}^2F(z,\bar{z},\lambda^{-1})R_{\lambda}^{-1}P_3\cdot P_3k(z,\bar{z})\cdot R_{\lambda}^{-1},
\end{split}\]
for $P_3=\hbox{diag}(1,1,1,-1,1,-1)$ and $P_3 k(z,\bar{z})\in SO^+(1,3)\times SO(2)$. Since \[R_{\lambda} P_3k(z,\bar{z})=P_3k(z,\bar{z}) R_{\lambda}  \hbox{ and } R_{\lambda}P_3=P_3R_{\lambda}^{-1},\] we obtain
\[F(\mu(z),\overline{\mu(z)},\lambda)=R_{\lambda}^2F(z,\bar{z},\lambda^{-1}) k(z,\bar{z}) .\]
Therefore we have $M(\lambda)=R_{\lambda}^2$. Note that  $M(\lambda)|_{\lambda=1}=I$ holds automatically now.

To compute \eqref{eq-mu-CF}, it will be convenient to transform $C(z,\lambda)$ into an upper triangular matrix, which has been used in \cite{Wa3}, \cite{Wa4}  and \cite{Wang-S4}. Here we retain the notation of \cite{Wang-S4}.
Let $F_-(z,\lambda)=\mathcal{P}(C)$ and $\chi(\lambda)=\mathcal{P}(M(\lambda))$. Then
\[F_-(\lambda)=\left(
                               \begin{array}{ccccccccc}
                                 1&   0& 0&0 &0&0&0&0 \\
                                  0&  1&  \lambda^{-1}f_1&\lambda^{-1}f_2 &\lambda^{-1}f_3&\lambda^{-1}f_4&\lambda^{-2}g_3&0 \\
                                   0& 0&1& 0 &0&0&-\lambda^{-1}f_4&0 \\
                                 0& 0&0 &1&   0&0&-\lambda^{-1}f_3&0 \\
                                  0& 0&0 &0& 1& 0&-\lambda^{-1}f_2&0 \\
                                   0& 0&0 &0&0&1& -\lambda^{-1}f_1&0 \\
                                   0& 0&0 &0&0&0& 1&0 \\
                                 0&   0& 0&0 &0&0&0&1 \\
                               \end{array}
                             \right)\]
                             with $g_3=-f_1f_4-f_2f_3$ and
\[\chi(\lambda)=\left(
                               \begin{array}{ccccccccc}
                                 1&   0& 0&0 &0&0&0&0 \\
                                 0 & \lambda^{-2}  & 0&0 &0&0&0&0 \\
                                   0& 0&1& 0 &0&0&0&0 \\
                                 0& 0&0 &1&   0&0&0&0 \\
                                  0& 0&0 &0& 1& 0&0&0 \\
                                   0& 0&0 &0&0&1& 0&0 \\
                                   0& 0&0 &0&0&0& \lambda^2&0 \\
                                 0&   0& 0&0 &0&0&0&1 \\
                               \end{array}
                             \right).\]
Moreover, we obtain                             $$\chi(\lambda)^{-1}F_-(\mu(z),\lambda)=\left(
                               \begin{array}{ccccccccc}
                                 1&   0& 0&0 &0&0&0&0 \\
                                  0& \lambda^2&  \lambda \hat{f}_1&\lambda \hat{f}_2 &\lambda \hat{f}_3&\lambda \hat{f}_4& \hat{g}_3&0 \\
                                   0& 0&1& 0 &0&0&-\lambda^{-1}\hat{f}_4&0 \\
                                 0& 0&0 &1&   0&0&-\lambda^{-1}\hat{f}_3&0 \\
                                  0& 0&0 &0& 1& 0&-\lambda^{-1}\hat{f}_2&0 \\
                                   0& 0&0 &0&0&1& -\lambda^{-1}\hat{f}_1&0 \\
                                   0& 0&0 &0&0&0& \lambda^{-2}&0 \\
                                 0&   0& 0&0 &0&0&0&1 \\
                               \end{array}
                             \right)$$
and
\[\begin{split}\mathcal{P}(\overline{C(z,\lambda^{-1})}^{-1})&=\check{J}_8\overline{F_-(\lambda^{-1})}^t\check{J}_8\\
&= \left(
                               \begin{array}{ccccccccc}
                                 1&   0& 0&0 &0&0&0&0 \\
                                  0&  1&0&0&0&0&0&0 \\
                                   0&  \lambda^{-1} \bar{f}_1&1& 0 &0&0&0&0 \\
                                  0& \lambda^{-1} \bar{f}_3&0 &1& 0& 0& 0&0 \\
                                 0& \lambda^{-1} \bar{f}_2 &0 &0&  1&0&0&0 \\
                                   0& \lambda^{-1} \bar{f}_4&0 &0&0&1& 0&0 \\
                                   0& \lambda^{-2} \bar{g}_3&-\lambda^{-1} \bar{f}_4 &-\lambda^{-1}\bar{f}_2&-\lambda^{-1}\bar{f}_3& -\lambda^{-1}\bar{f}_1& 1&0 \\
                                 0&   0& 0&0 &0&0&0&1 \\
                               \end{array}
                             \right).\end{split}\]
Here $\hat{f}_j=f_j\circ\mu,$ $j=1,2,3,4$, and $\hat{g}_3=g_3\circ\mu$.

Finally we obtain that  $\mathcal{P}(\overline{C(z,\lambda^{-1})}^{-1})\mathcal{P}(M(\lambda)^{-1})\mathcal{P}(C(\mu(z),\lambda))$ is of the form
\[\left(
                               \begin{array}{ccccccccc}
                                 1&   0& 0&0 &0&0&0&0 \\
                                  0& \lambda^2&  \lambda \hat{f}_1&\lambda \hat{f}_2 &\lambda \hat{f}_3&\lambda \hat{f}_4&\hat{g}_3 &0 \\
                                  0& \lambda \bar{f}_1& 1+\bar{f}_1\hat{f}_1&  \bar{f}_1\hat{f}_2 &  \bar{f}_1\hat{f}_3&  \bar{f}_1 \hat{f}_4&  \lambda^{-1}(\bar{f}_1\hat{g}_3-\hat{f}_4)&0 \\
                                  0& \lambda \bar{f}_3& \bar{f}_3\hat{f}_1& 1+\bar{f}_3 \hat{f}_2 &  \bar{f}_3\hat{f}_3&  \bar{f}_3 \hat{f}_4&  \lambda^{-1}(\bar{f}_3\hat{g}_3-\hat{f}_3)&0 \\
                                  0& \lambda \bar{f}_2& \bar{f}_2\hat{f}_1& \bar{f}_2 \hat{f}_2 & 1+ \bar{f}_2\hat{f}_3&  \bar{f}_2 \hat{f}_4&  \lambda^{-1}(\bar{f}_2\hat{g}_3-\hat{f}_2)&0 \\
                                  0& \lambda \bar{f}_4&  \bar{f}_4 \hat{f}_1&  \bar{f}_4\hat{f}_2 &  \bar{f}_4\hat{f}_3& 1+\bar{f}_4\hat{f}_4&  \lambda^{-1}(\bar{f}_4\hat{g}_3-\hat{f}_1)&0 \\
                                   0& \bar{g}_3& w_{73}&w_{74} &w_{75} & w_{76} & w_{77}&0 \\
                                 0&   0& 0&0 &0&0&0&1 \\
                               \end{array}
                             \right)\]
                    with
\[\begin{split}&w_{73}=\lambda^{-1}(\bar{g}_3\hat{f}_1-\bar{f}_4),\ w_{74}=\lambda^{-1}(\bar{g}_3\hat{f}_2-\bar{f}_2),\
 w_{75}=\lambda^{-1}(\bar{g}_3\hat{f}_3-\bar{f}_3),\ w_{76}=\lambda^{-1}(\bar{g}_3\hat{f}_4-\bar{f}_1),\\
&w_{77}=\lambda^{-2}(1+\bar{g}_3\hat{g}_3+\bar{f}_4\hat{f}_4+\bar{f}_2\hat{f}_3+\bar{f}_3\hat{f}_2+\bar{f}_1\hat{f}_1).
\end{split}\]
Therefore, \eqref{eq-mu-CF} is equivalent to
$$\bar{g}_3\hat{f}_1-\bar{f}_4=\bar{g}_3\hat{f}_3-\bar{f}_3=\bar{g}_3\hat{f}_2-\bar{f}_2=\bar{g}_3\hat{f}_4-\bar{f}_1=0,$$
$$\bar{f}_1\hat{g}_3-\hat{f}_4=\bar{f}_2\hat{g}_3-\hat{f}_2=\bar{f}_3\hat{g}_3-\hat{f}_3=\bar{f}_4\hat{g}_3-\hat{f}_1=0,$$
and
$$1+g_3(z)\overline{{g}_3(\mu(z))}+f_1(z)\overline{{f}_1(\mu(z))}+f_2(z)\overline{{f}_3(\mu(z))}+f_3(z)\overline{{f}_2(\mu(z))}+f_4(z)\overline{f_4(\mu(z))}=0.$$
Since $g_3=-f_1f_4-f_2f_3$, one verifies that these equations are equivalent to \eqref{eq-iso-rp2}.

Then (1) and (2-a) of Theorem \ref{thm-rp2} holds.

\end{proof}

\section{Amendment and correction  to:   ``On symmetric Willmore surfaces in spheres I: the orientation preserving case"[Differ.   Geom.   Appl.  43 (2015) 102-129]}

$(\mathrm{A})$. In \cite{DoWa-sym1} we consider  Willmore surfaces $y:M\rightarrow S^{n+2}$
and symmetries  $(\gamma,R)$ defined by the relation:
\[y(\gamma(p))=Ry(p),\]
with $\gamma:M\rightarrow M$ a conformal transformation of $M$ and $R$ a conformal transformation of $S^{n+2}$.
In this case we then obtain $f(\gamma(p))= \hat Rf(p)$, where $\hat R$ denotes the element of $O^+(1,n+3)$ which acts on the light cone containing $S^{n+2}$ and which represents $R$
after projection to $S^{n+2}$.

In $(4.3)$, the statement concerning S-Willmore surfaces in Theorem 4.1, page 109 of \cite{DoWa-sym1}, we stated \[\hbox{  $f(\gamma(p))=\hat{f}(\gamma(p))=\hat Rf(p)$}.\]
Here $\hat f$ denotes the conformal Gauss map of $\hat y$,  the dual surface of $y$. We did not discuss orientations then. But, since in the present paper we have addressed the issue of orientations in detail, we would like to point out:

Since for an S-Willmore surface $y$  with a non-degenerate dual surface $\hat y$, their conformal Gauss map maps to the same 4--dimensional Lorentzian subspace, but with opposite orientations \cite{Ma,DoWa1}, using our present notation
(see \eqref{eq-star-op} for the definition of $\star$) we should write:
\[f(\gamma(p))= \hat Rf(p)^\star.\]
So for a symmetry we have (as before)  the following transformation formulas for $y$
\begin{equation}\label{eq-correct-y}
 y(\gamma(p))=Ry(p) \hbox{ if and only if }f(\gamma(p))= \hat Rf(p)
\end{equation}
and  for a transformation from $y$ to its non-degenerate dual surface $\hat y$ we have the formulas
 \begin{equation}\label{eq-correct-haty}
 \hat{y}(\gamma(p))=Ry(p), \hbox{ if and only if }(f(\gamma(p)))= \hat Rf(p)^\star.
\end{equation}

  This notational issue only occurs in Theorem 4.1 and Theorem 4.11 of \cite{DoWa-sym1}.\vspace{2mm}

 $(\mathrm{B})$. We would also like to take this opportunity to correct some factors occurring in Example 5.3 of \cite{DoWa-sym1} :
In this example,  we considered examples of branched isotropic Willmore surfaces with positive genus.  To obtain the condition $f_1'f_4'+f_2'f_3'=0$, the meromorphic differentials in (5.4) of \cite{DoWa-sym1} should satisfy the relations:
\[\dd f_1=\dd \mathbf{f},\ \dd f_2=df_3=\mathbf{h}\dd\mathbf{f},\ \dd f_4=-\mathbf{h}^2\dd\mathbf{f}.\]
Here $\dd\mathbf{f}$, $\mathbf{h}\dd\mathbf{f}$ and $\mathbf{h}^2\dd\mathbf{f}$ are linearly independent meromorphic 1-forms on a compact Riemann surface $M$ with no residues and no periods (see the proof of the main theorem in \cite{Yang}).
And the equation (5.5) of \cite{DoWa-sym1} should be
\begin{equation}\label{eq-meromorphic functions-r3-deform}
\dd f_1=\dd\mathbf{f},\ \dd f_2=\mathbf{h} \dd\mathbf{f},\ \dd f_3=t_0\mathbf{h}\dd\mathbf{f},\ \dd f_4=-t_0\mathbf{h}^2\dd\mathbf{f},\end{equation}
$t_0\in\C^*$.
Then the resulting (branched) Willmore surface $y$ is globally defined on $M$. Moreover, by Theorem 1.3 of \cite{Wa4}, we have:\begin{enumerate} \item $y$ is not  conformally equivalent to any minimal surface in any space form when
$ t_0\not\in\R$;\item $y$ is conformally equivalent to some minimal surface in $S^4$ when $t_0>0$; \item$y$ is  conformally equivalent to some minimal surface in $\mathbb{H}^4$ when $t_0<0$ .
\end{enumerate}
\vspace{3mm}

{\small{\bf Acknowledgements} The authors are thankful to Prof. David Brander for the helps in the drawing of pictures.
Part of this work was done when the second named author visited the Department of Mathematics of TU-M\" unchen. He would like to express his sincere gratitude for both the hospitality and financial support. The second named author was partly supported by the Project 11571255 and 11831005 of NSFC.
 Some of the visits were supported partially by a grant  for the Project Centered Exchange Program between the State of  Bavaria and  P.R. China.

{\footnotesize
\def\refname{Reference}

}
\vspace{2mm}
{\footnotesize \begin{multicols}{2}   
Josef F. Dorfmeister

Fakult\" at f\" ur Mathematik,

TU-M\" unchen, Boltzmann str. 3,

D-85747, Garching, Germany

{\em E-mail address}: dorfm@ma.tum.de\\

Peng Wang

College of Mathematics \& Informatics, FJKLMAA,

Fujian Normal University, Qishan Campus,

Fuzhou 350117, P. R. China

{\em E-mail address}: {pengwang@fjnu.edu.cn}

\end{multicols}}


\begin{thebibliography}{0}

\bibitem{BB} Babich, M., Bobenko, A.{\em  Willmore tori with umbilic lines and minimal surfaces in hyperbolic space,} Duke Math. J, 1993, 72(1): 151-185.

\bibitem{Brander}Brander, D. {\em Bj\"{o}rling's Problem and DPW for Willmore Surfaces},  \\ http://davidbrander.org/Images/WillmoreSurfaces.html.
\bibitem {Bryant1982} Bryant, R. {\em Conformal and minimal immersions of compact surfaces into the 4-sphere, } J.
Diff. Geom. 17(1982), 455-473.

\bibitem {Bryant1984} Bryant, R. {\em A duality theorem for Willmore surfaces,} J. Diff. Geom. 20(1984), 23-53.

\bibitem{BuGu} Burstall,  F.E., Guest, M.A., {\em Harmonic two-spheres in compact symmetric spaces, revisited,} Math. Ann. 309 (1997), 541-572.

\bibitem{BPP} Burstall, F., Pedit, F., Pinkall, U.  {\em  Schwarzian derivatives and flows of surfaces,} Contemporary Mathematics 308, 39-61, Providence, RI: Amer. Math. Soc.,  2002.

\bibitem{DPW} Dorfmeister, J., Pedit, F., Wu, H., {\em Weierstrass type representation of harmonic maps into symmetric spaces,} Comm. Anal. Geom. 6 (1998), 633-668.

\bibitem{DoWa1} Dorfmeister, J.,  Wang, P., {\em Willmore surfaces in spheres via loop groups I: generic cases and some examples}, arXiv:1301.2756v4.


\bibitem{DoWa11} Dorfmeister, J.,  Wang, P., {\em Willmore surfaces in spheres: the DPW approach via the conformal Gauss map.}
 Abh. Math. Semin. Univ. Hambg. 89 (2019), no. 1, 77-103.
\bibitem{DoWa-sym1} Dorfmeister, J.,  Wang, P., {\em On symmetric Willmore surfaces in spheres I: the orientation preserving case},  Diff. Geom. Appl.,  Vol. 43 (2015),  102-129.

\bibitem{Do-Wa-equ} Dorfmeister, J.,  Wang, P., {\em On equivariant Willmore surfaces in $S^{n+2}$}, in preparation.


\bibitem{Ejiri-equ} Ejiri N., {\em Equivariant minimal immersions of $S^2$ into $S^{2m}(1)$,} Trans. Amer. Math. Soc., 1986, 297(1): 105-124.

 \bibitem {Ejiri1988} Ejiri, N. {\em  Willmore surfaces with a duality in $S^{n}(1)$,} Proc. London Math. Soc. (3), 1988, 57(2),
 383-416.

\bibitem{Eujalance} Eujalance, E, Javier Cirre, F., Gamboa, J.M., Gromadzki, G. {\em Symmetries of Compact Riemann surfaces}, LMN, Springer, 2010.
 \bibitem {Helein} H\'{e}lein, F. {\em Willmore immersions and loop groups,} J. Differential Geom. 50 (1998),
no. 2, 331-385.
\bibitem{Ishihara} Ishihara T. {\em Harmonic maps of nonorientable surfaces to four-dimensional manifolds,} Tohoku Math. J., 1993, 45(1): 1-12.

 

\bibitem{Ma} Ma, X. {\em  Willmore surfaces in $S^{n}$: transforms and
vanishing theorems,} dissertation, Technische Universit\"{a}t
Berlin, 2005.

\bibitem{Mon} Montiel, S. {\em  Willmore two spheres in the four-sphere,} Trans. Amer. Math. Soc., 2000, 352(10), 4469-4486.

\bibitem{Wa3}  Wang, P., {\em  Willmore surfaces in spheres via loop groups II: a coarse classification of Willmore two-spheres by potentials}, arXiv:1412.6737, submitted.

\bibitem{Wa4} Wang, P., {\em Willmore surfaces in spheres via loop groups III:  on minimal surfaces in space forms}, Tohoku Math. J. (2) 69 (2017), no. 1, 141-160.
 \bibitem{Wang-S4}   Wang, P., {\em A Weierstrass type representation of isotropic Willmore surfaces in $S^4$}, in preparation.


\bibitem{Wu} Wu, H.Y. {\em  A simple way for determining the normalized potentials for harmonic maps,} Ann.
Global Anal. Geom.  17 (1999), 189-199.

\bibitem{Yang} Yang, K. {\em Meromorphic functions on a compact Riemann surface and associated complete minimal surfaces}, Proc. Amer. Math. Soc., 1989, 105(3), 706-711.

\end{thebibliography}
\end{document}